\documentclass[11pt]{amsart}

\usepackage{amssymb}
\usepackage{graphics}
\usepackage{graphicx}
\usepackage{placeins}
\usepackage{hyperref} 

\newtheorem{theorem}{Theorem}

\newtheorem{lemma}{Lemma}
\newtheorem{prop}{Proposition}

\theoremstyle{definition}
\newtheorem{defn}{Definition}

\newtheorem{cor}{Corollary}

\theoremstyle{definition}

\hypersetup{
    colorlinks=true,
    citecolor=cyan, % Choose your preferred color
    linkcolor=blue
}

\begin{document}

\author{Sang Woo Yoon}
\title{Khovanov homology and roll-spun slice disks}

\begin{abstract}
    We show that Khovanov homology cannot distinguish the roll-spun slice disk from the trivial slice disk bounding the connected sum of a knot and its mirror, when composed with a Morse $1$-handle.
\end{abstract}

\maketitle

\section{Introduction}

Khovanov homology \cite{Kh00} is a link invariant for links in $S^3$ that categorifies the Jones polynomial. It is a functorial invariant \cite{Ja04, MWW22} --- oriented link cobordisms induce maps between Khovanov homologies (see Section 2 for a precise statement).

Functoriality of Khovanov homology can be exploited to distinguish two smoothly embedded slice disks in $D^4$ that are topologically isotopic \cite{HS24}. On the other hand, \cite{JZ20} showed knot Floer homology $\widehat {HFK}$ can distinguish a certain class of slice disks, the \textit{roll-spun slice disks}, from the trivial slice disk bounding the connected sum $K\#\bar K$ of a knot $K$ and its mirror. More precisely, they showed that cobordism maps on $\widehat{HFK}$ can distinguish the roll-spun slice disk from the trivial slice disk for a certain class of knots. We may ask whether Khovanov homology can also distinguish these slice disks, at least for some knots. In this paper, we give a partial negative result in this direction.

\begin{theorem}
    Let $K$ be a knot and $\bar K$ its mirror. Let $D_0$, $D_\text{roll}$ be the trivial and the roll-spun slice disk bounding $K\# \bar K$ respectively. Let $H$ be the cobordism $K\sqcup \bar K\to K\#\bar K$ given by the $1$-handle attachment. We have 
    \[
        Kh(D_0\circ H)=\pm Kh(D_\text{roll}\circ H).
    \]
    That is, Khovanov homology fails to distinguish $D_0\circ H$ and $D_\text{roll}\circ H$ viewed as cobordisms $K\sqcup \bar K\to \varnothing$ up to signs.
\end{theorem}

In particular, if $K$ is a knot such that $Kh(H) : Kh(K\sqcup \bar K)\to Kh(K\#\bar K)$ is surjective, then Khovanov homology does not distinguish the roll-spun slice disk from the trivial slice disk when viewed as cobordisms $K\#\bar K\to \varnothing$.

The paper is organized as follows. In Section 2, we recall the definition of Khovanov homology and give a precise statement concerning functoriality. In Section 3, we reproduce the definition of roll-spun slice disks in \cite{JZ20}, and construct explicit movie decompositions of the slice disks. In Section 4, we state the main theorem, and give an overview of its proof. The details of the proof can be found in the appendix.

\section{Khovanov homology}

Let $D$ be a diagram of an oriented link $L$, together with an ordering on its crossings. The\textit{Khovanov chain complex} $CKh(D)$ associated to this data is a chain complex bigraded by the \textit{homological grading} $\text{gr}_h$ and \textit{quantum grading} $\text{gr}_q$. 

The complex $CKh(D)$ is defined as follows. We reproduce the concise description given in \cite{HS24}. A crossing $c$ of $D$ can be resolved either by a \textit{0-smoothing} or a \textit{1-smoothing}, as described in Figure \ref{01smoothing}.

    \begin{figure}[!htbp]
    \centering
    \includegraphics[width=0.3\textwidth]{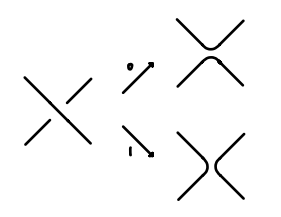}
    \label{01smoothing}
    \end{figure}
    \FloatBarrier

A \textit{smoothing} of $D$ is a closed 1-manifold in the plane obtained by resolving each crossing of $D$ either by a 0 or 1-smoothing. To each smoothing of $D$, we can associate a tuple $\sigma =(\sigma_1,\dots,\sigma_n)$ with entries either 0 or 1, where $\sigma_i$ indicates how the $i$-th crossing is resolved. A \textit{labeled smoothing} of $D$ is a smoothing of $D$ where every component is labeled with a $1$ or $x$. 

The chain group $CKh(D)$ is defined to be the free abelian group generated by all labeled smoothings of $D$. Let $a_\sigma$ be a labeled smoothing, with $\sigma$ the binary tuple associated to the smoothing underlying $a$. The homological grading of $a_\sigma$ is given by
\[
\text{gr}_h(a_\sigma) = |\sigma| - n_-,
\]
where $|\sigma| = \sum_i \sigma_i$, and $n_-$ is the number of negative crossings in $D$. The quantum grading of $a_\sigma$ is given by
\[
\text{gr}_q(a_\sigma) = |a_\sigma| + \text{gr}_h(a_\sigma) + n_+ - n_-
\]
where $n_+$ is the number of positive crossings in $D$ and $|a_\sigma|$ is the difference
\[
(\text{number of components in } a_\sigma \text{ labeled 1}) - (\text{number of components in } a_\sigma \text{ labeled } x).
\]
The differential $d : CKh(D)\to CKh(D)$ is a map of bidegree $(\text{gr}_h, \text{gr}_q) = (1,0)$ defined by the formula
\[
d(a_\sigma) = \sum_{\{i\mid\sigma_i = 0\}} (-1)^{\xi_i} a_{\sigma^i},
\]
where $\sigma^i$ is a binary $n$-tuple such that $\sigma^i_j = \sigma_j$ for all $j\neq i$ and $\sigma^i_i =1$, and $\xi_i = \sum_{j<i}\sigma_j$. The labeled smoothing $a_{\sigma^i}$ is defined as follows. The labels of $a_{\sigma^i}$ away from the $i$-th crossing agrees with that of $a_\sigma$. By changing the resolution of the $i$-th crossing from a $0$-smoothing to an $1$-smoothing we either merge two components, or split a component into two. The labels of $a_{\sigma^i}$ at these components are assigned by the following rules.

    \begin{figure}[!htbp]
    \centering
    \includegraphics[width=0.5\textwidth]{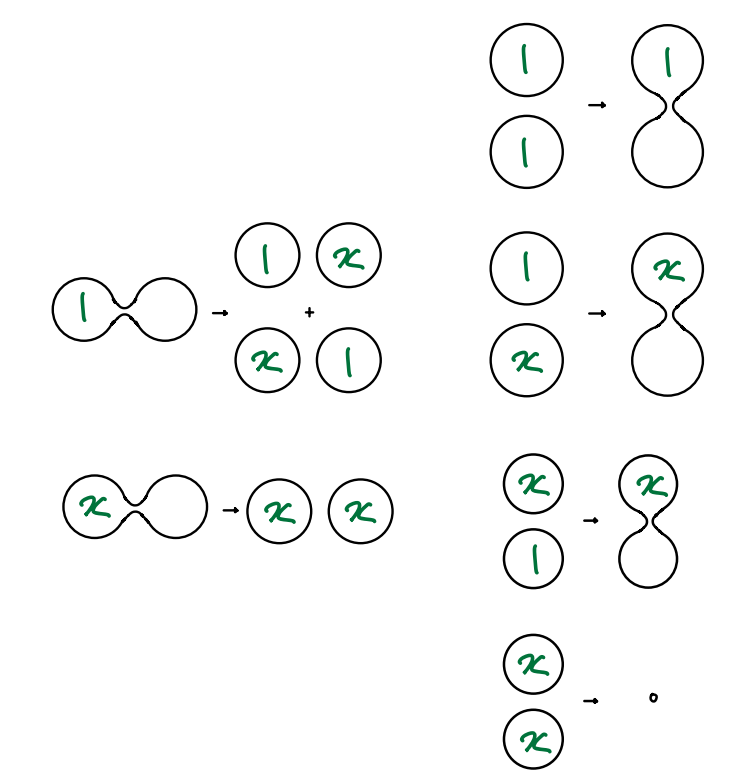}
    \label{mult}
    \end{figure}
    \FloatBarrier

The resulting homology $Kh(D)$ is called the \textit{Khovanov homology} of $D$. This group is bigraded by $\text{gr}_{h}$ and $\text{gr}_q$. Furthermore, Reidemeister moves between planar link diagrams induce quasi-isomorphisms between the associated Khovanov chain complexes. Thus $Kh(D)$ is a link invariant.

\subsection{Cobordism-induced chain map}
Let $L_0$ and $L_1$ be oriented links in $\mathbb R^3\times \{0\}$ and $\mathbb R^3\times \{1\}$ respectively. A \textit{link cobordism} $\Sigma : L_0 \to L_1$ is a compact, oriented, smoothly embedded surface $\Sigma\hookrightarrow \mathbb R^3\times [0,1]$ such that $\partial\Sigma = L_0\cup L_1$. 

After an isotopy (fixing the boundary set-wise), a link cobordism $\Sigma : L_0\to L_1$ can be broken down into a \textit{movie}, which is a sequence $D_0 = D_{t_0}, D_{t_1},\dots, D_{t_m} = D_1$ of oriented link diagrams such that consecutive diagrams $D_{t_i}, D_{t_{i+1}}$ are related either by a planar isotopy, a Reidemeister move, or a Morse move. Two movies represent isotopic cobordisms if and only if they are related by \textit{Carter-Saito movie moves}. To each Reidemeister and Morse moves, we can associate a chain map on the Khovanov chain complexes. Hence a movie $D_0 \to D_1$ induces a chain map $CKh(D_0)\to CKh(D_1)$. \cite{Ja04} shows that these chain maps are invariant under the Carter-Saito movie moves up to an overall sign, thereby proving the following theorem.

\begin{theorem}[\cite{Ja04}]
    Let $\Sigma,\Sigma' : L_0\to L_1$ be two link cobordisms related by a smooth isotopy fixing $L_0\sqcup L_1$ set-wise. Then the induced maps $Kh(\Sigma),Kh(\Sigma'):Kh(L_0)\to Kh(L_1)$ agree up to an overall sign.
\end{theorem}

Since we will be working with slice disks in $B^4$ with boundary in $S^3 = \partial B^4$, we need a version of the above result for links in $S^3$. This was proven in \cite{MWW22}, by showing that Khovanov homology satisfies the \textit{sweep-around move}.

In this paper we are concerned with movies involving only Reidemeister II moves and Morse $1$-handle moves. Let $D_0$ and $D_1$ be diagrams related by a single Morse $1$-handle move. Then the associated chain map $\mu_1:CKh(D_0)\to CKh(D_1)$ is defined on labeled smoothings as follows. The $1$-handle either merges two components, or splits one component into two. On components away from the $1$-handle, $\mu_1$ keeps the labels unchanged. On the components being merged or split, $\mu_1$ is defined by the (co)multiplication maps given above.

There are two chain maps associated to a Reidemeister II move --- the chain map $\rho$ induced by the Reidemeister II move resolving two crossings, and $\rho'$ induced by the reverse move creating two crossings. The effect of these chain maps are described below. We follow the descriptions given in \cite{Ja04}, \cite{Kh00}, \cite{HS24}.

    \begin{figure}[!htbp]
    \centering
    \includegraphics[width=0.5\textwidth]{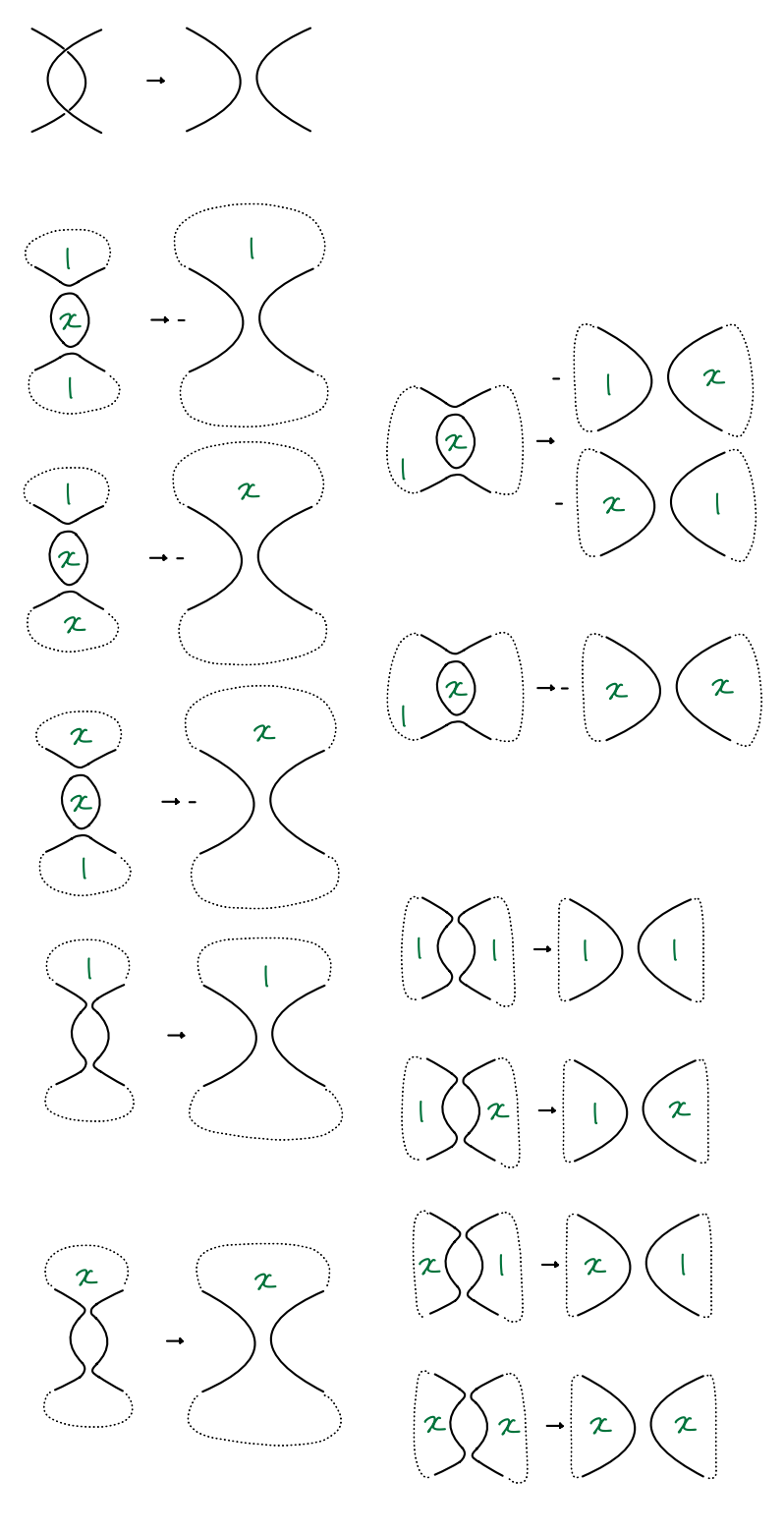}
    \caption{The chain map $\rho$ induced by the Reidemister II move.} 
    \label{reidemeister}
    \end{figure}
    \FloatBarrier

    \begin{figure}[!htbp]
    \centering
    \includegraphics[width=0.5\textwidth]{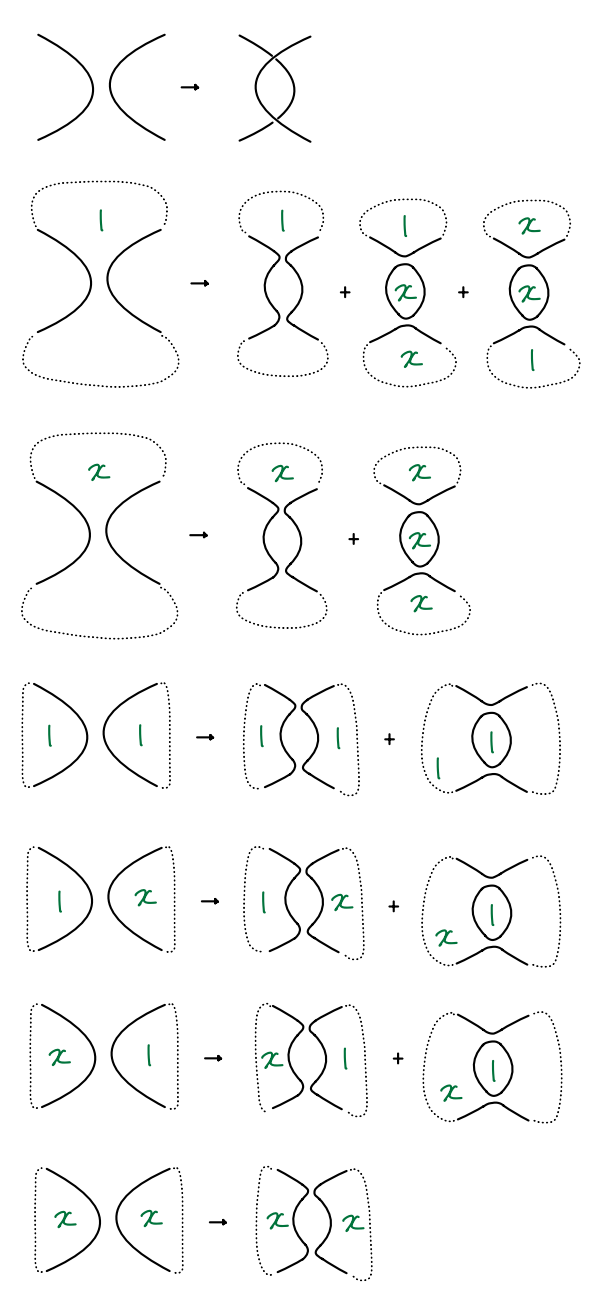}
    \caption{The chain map $\rho'$ induced by the reverse Reidemister II move.} 
    \label{reidemeister-1}
    \end{figure}
    \FloatBarrier

The above figures should be interpreted as follows. The chain maps $\rho$ and $\rho'$ keep the labels of components outside the corresponding local pictures unchanged. Any configuration of labeled smoothings not described in the above figures are mapped to zero.

\section{Roll-spun slice disk}

Let $K$ be a knot in $S^3$, and let $\bar K$ denote its mirror. \textit{Roll-spinning}, and more generally \textit{deform-spinning} due to \cite{Li79}, produce slice disks in $D^4$. We reproduce the construction given in \cite{JZ20}.

Let $B$ be an open 3-ball of $S^3$ such that $\bar B\cap K$ is an unknotted arc. Identify $S^3\setminus B \cong D^3$, and let $a$ be the (knotted) arc $K\cap (S^3\setminus B)$. Given an isotopy $\varphi: D^3\times I\to D^3$ such that $\varphi_0 = \text{id}_{D^3}$, $\varphi_t|_{\partial D^3} = \text{id}_{\partial D^3}$ for all $t\in I$ and $\varphi_1(a) = a$, the \textit{deform-spun slice disk} $D_{a,\varphi}$ in $D^4$ bounding $K\#\bar K$ is defined by taking
\[
    \bigcup_{t\in I} \varphi_t(a)\times \{t\} \subset D^3\times I
\]
and rounding the corners. In particular, deform-spinning $K$ using the trivial isotopy $\varphi_t = \text{id}_{D^3}$ yields the trivial slice disk bounding $K\#\bar K$. We denote this disk by $D_0$. 

\begin{lemma}[Lemma 3.3, \cite{JZ20}]
    Let $d$ be an automorphism of $(D^3,a)$ such that $d|_{\partial D^3} = \text{id}_{\partial D^3}$. Then there is an isotopy $\varphi : D^3\times I$ such that $\varphi_0 = \text{id}_{D^3}$, $\varphi_1 = d$. Furthermore, the isotopy class of $D_{a,\varphi}$ depends only on $d$.
\end{lemma}

Hence we may denote $D_{a,\varphi}$ by $D_{a,d}$. Roll-spinning can be viewed as a special case of deform-spinning, earned by choosing $d$ as follows. 

\begin{defn}[Definition 3.5, \cite{JZ20}]
    Let $\nu K\cong K\times D^2$ be a tubular neighborhood of $K$ in $S^3$. Let $X = S^3\setminus \text{int}(\nu K)$ be the knot exterior. Let $\partial X\times I$ be the collar of $\partial X$ contained in $X$. Identifying $K\cong \mathbb R/\mathbb Z$, choose a smooth monotonically increasing function $\varphi : \mathbb R\to I$ with $\varphi(t) = 0$ for $t\leq 0$ and $\varphi(t) =1$ for $t\geq 1$. Define the diffeomorphism $r : (S^3,K)\to (S^3,K)$ by 
    \[
    r(\bar x, \bar\theta, t) = (\overline{x+\varphi(t)}, \bar \theta, t)
    \]
    for $(\bar x, \bar\theta, t)\in K\times \partial D^2 \times I\cong \partial X\times I$, and setting $r$ to be the identity outside the collar $\partial X\times I$.
\end{defn}

\begin{defn}
    Let $K$ be a knot in $S^3$, and let $B$ be an open 3-ball of $S^3$ such that $\bar B\cap K$ is an unknotted arc. The deform-spun slice disk $D_{a,r}$ is the \textit{roll-spun slice disk} bounding $K\#\bar K$ in $D^4$, which we denote $D_\text{roll}$.
\end{defn}

\subsection{Movie decomposition of $D_\text{roll}$}

Let $K$ be an oriented knot in $S^3$. Choose a planar diagram of $K$, which by abuse of notation we also denote by $K$. 

Let $D_0$ and $D_\text{roll}$ be the trivial and roll-spun slice disks bounding $K\#\bar K$ respectively, viewed as cobordisms $K\#\bar K\to \varnothing$. Let $R$ denote the cobordism from $K\#\bar K$ to itself obtained as follows. First, shrink a ball containing $\bar K$ and a small neighborhood of the attaching region for $K\#\bar K$, small enough such that planar translations of the ball fit into the local pictures of $K$ around each crossing. Move the ball together with the attaching region along $K$ in accordance with the orientation of $K$. Once the ball has traversed around $K$ once, enlarge the ball to its original position. The following observation was noted in \cite{JZ20}.

\begin{lemma}
    $D_\text{roll}$ is isotopic to $D_0\circ R$.
\end{lemma}

\begin{proof}
    Let $I$ be the cylinder cobordism $(K\#\bar K)\times [0,1]$. Then $D_\text{roll}\cong D_\text{roll}\circ I$. View this cobordism as lying in $D^4$, as illustrated in Figure \ref{CobD0}. By pushing the complicated portion of the cobordism towards one end of the cylinder as in Figure \ref{CobDroll}, we obtain $D_\text{roll}\circ I\cong D_0\circ R\circ I$.
\end{proof}

    \begin{figure}[!htbp]
    \centering
    \includegraphics[width=0.25\textwidth]{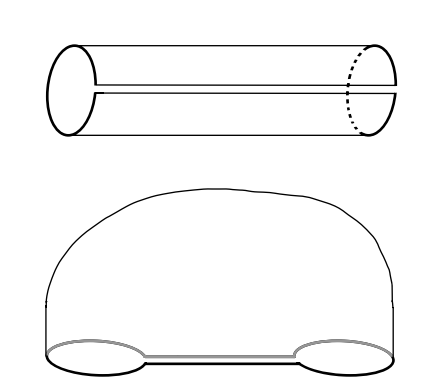}
    \caption{The trivial slice disk $D_0$ when $K$ is the unknot.} 
    \label{CobD0}
    \end{figure}
    \FloatBarrier

    \begin{figure}[!htbp]
    \centering
    \includegraphics[width=0.5\textwidth]{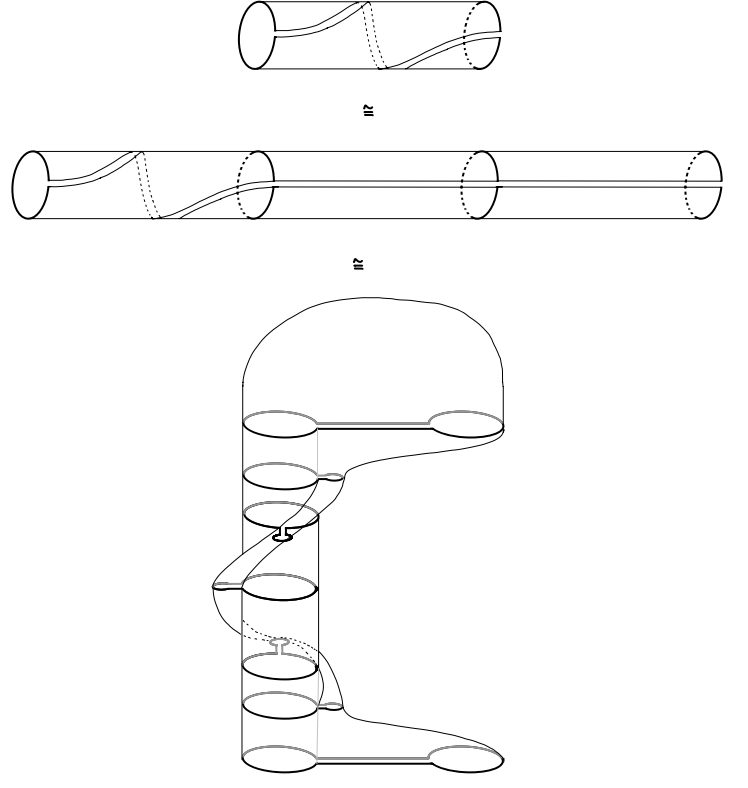}
    \caption{The roll-spun slice disk $D_\textbf{roll}$ when $K$ is the unknot. The isotopy between $D_\text{roll}$ and $D_0\circ R$ is described also.} 
    \label{CobDroll}
    \end{figure}
    \FloatBarrier

Let $H$ denote the cobordism obtained by performing a $1$-handle move at the band very near the attaching region for $\bar K$ in the connected sum. To further simplify the cobordisms in hand, we consider $D_0\circ H$ and $D_\text{roll}\circ H\cong D_0\circ R\circ H$. The latter cobordism can be broken down into simpler cobordisms $S$ and $E$, which admit Carter-Saito movie decompositions consisting only of Reidemeister II moves.

    \begin{figure}[!htbp]
    \centering
    \includegraphics[width=0.8\textwidth]{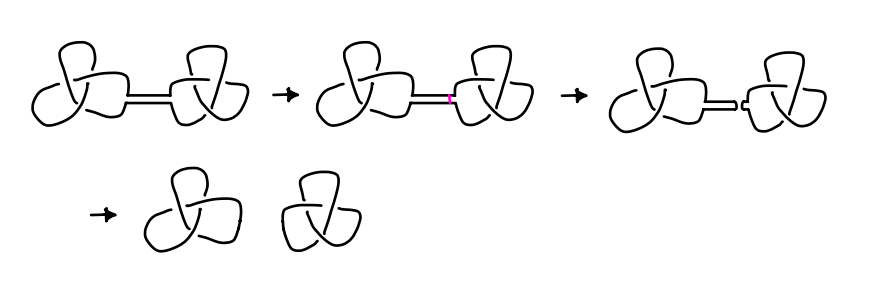}
    \caption{Movie decomposition of the 1-handle cobordism $H$.} 
    \label{CobH}
    \end{figure}
    \FloatBarrier

The cobordism $E : K\sqcup \bar K\to K'\sqcup \bar K$ is defined as follows. Here $K'$ is an equivalent diagram of $K$ which will be described below. Replace the attaching region for $K\#\bar K$ in $K$ with a bump protruding towards $\bar K$, which we call a \textit{tentacle}. Choose a ball containing $\bar K$ and the tip of the tentacle, and move the ball along $K$. Once the ball has traversed around $K$ once, we arrive at a link diagram that we denote by $K'\sqcup \bar K$. Whenever the tentacle encounters a strand of $K$, it moves under/over, depending on the under/over data on the corresponding crossing of $K$. Hence $E$ can be decomposed into a movie consisting only of Reidemeister II moves.

    \begin{figure}[!htbp]
    \centering
    \includegraphics[width=0.8\textwidth]{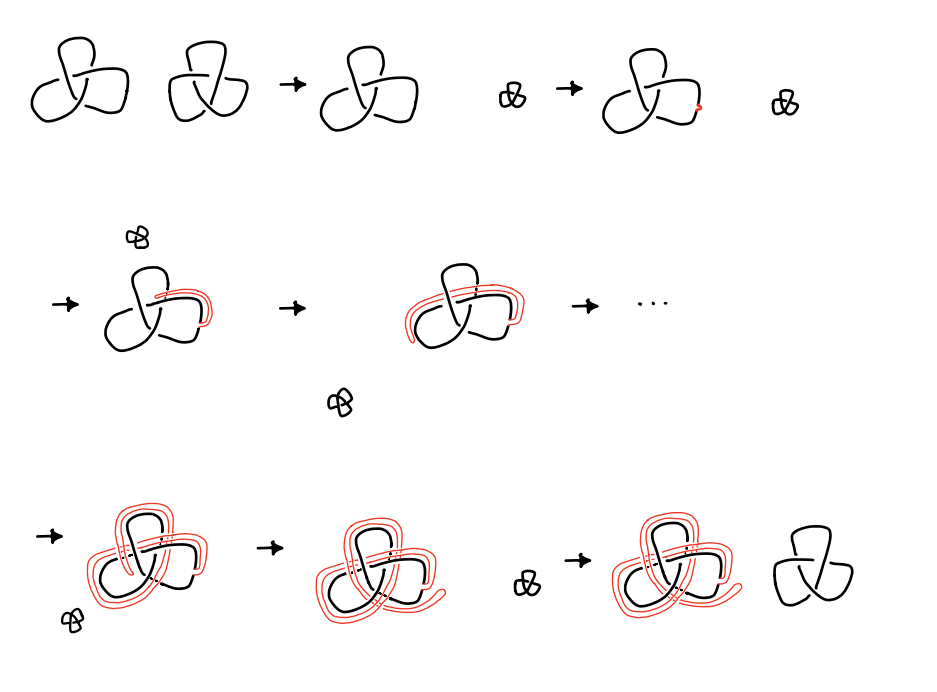}
    \caption{Movie decomposition of the cobordism $E$.} 
    \label{CobE}
    \end{figure}
    \FloatBarrier

The cobordism $H' : K'\sqcup \bar K\to K'\#\bar K$ is then obtained by performing a $1$-handle move between the tip of the tentacle, and the attaching region for $K\#\bar K$ in $\bar K$. 

    \begin{figure}[!htbp]
    \centering
    \includegraphics[width=0.75\textwidth]{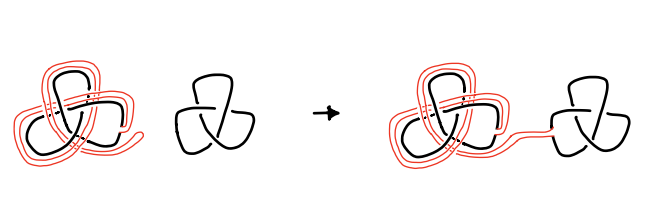}
    \caption{Movie decomposition of the 1-handle cobordism $H'$.} 
    \label{CobH_}
    \end{figure}
    \FloatBarrier

Finally, the cobordism $S : K'\#\bar K\to K\# \bar K$ is earned by moving the base of the tentacle along $K$. Once the base traverses around $K$ once, we retrieve $K\#\bar K$. Equivalently, $S$ is obtained by performing a series of Reidemeister II moves along the region formed between the tentacle and $K$, and then performing a planar isotopy.

    \begin{figure}[!htbp]
    \centering
    \includegraphics[width=0.8\textwidth]{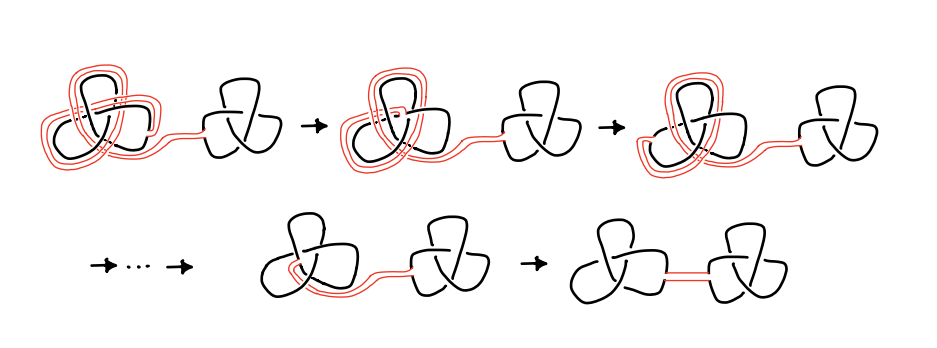}
    \caption{Movie decomposition of the cobordism $S$.} 
    \label{CobS}
    \end{figure}
    \FloatBarrier

\begin{lemma}
    $R\circ H$ is isotopic to $S\circ H'\circ E$.
\end{lemma}

\begin{proof}
    The moving ball in $R$ traces a tubular neighborhood of a curve in $S^3\times [0,1]$ that projects to a generator of $\pi_1(K)\cong \mathbb Z$. The isotopy $R\circ H\cong S\circ H'\circ E$ is obtained by unwinding this curve, at the cost of moving neighborhoods of $\bar K$ in the portion corresponding to $H$ towards the opposite orientation of the above generator. See Figures \ref{Cob1} and \ref{Cob2}.
\end{proof}

    \begin{figure}[!htbp]
    \centering
    \includegraphics[width=0.5\textwidth]{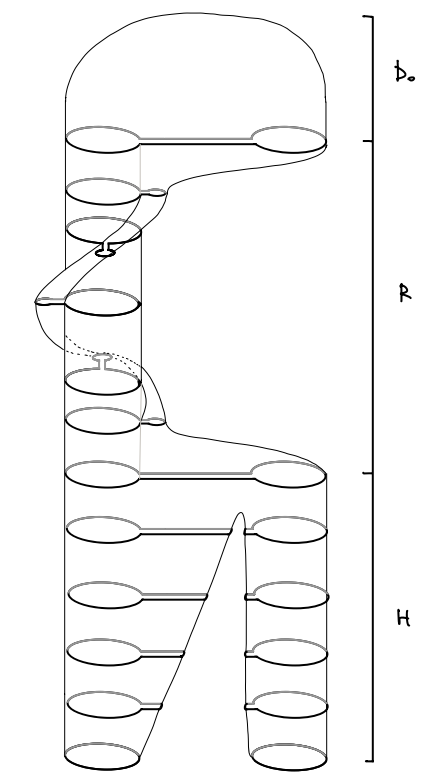}
    \caption{The cobordism $D_0\circ R\circ H$ when $K$ is the unknot.}
    \label{Cob1}
    \end{figure}
    \FloatBarrier

    \begin{figure}[!htbp]
    \centering
    \includegraphics[width=0.5\textwidth]{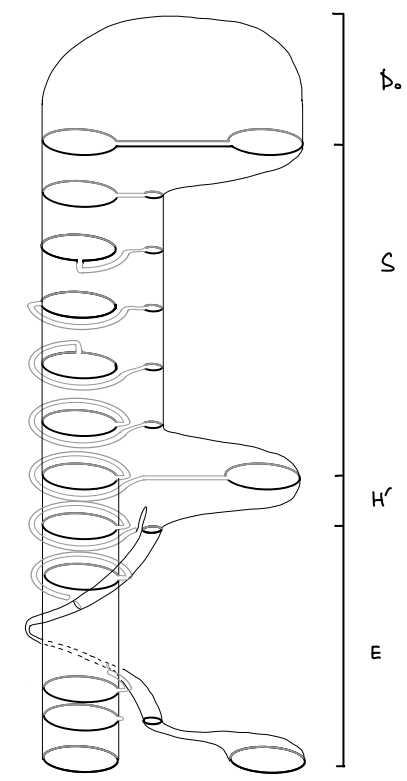}
    \caption{The cobordism $D_0\circ S\circ H'\circ E$ when $K$ is the unknot. The isotopy from $D_0\circ R\circ H$ to $D_0\circ S\circ H'\circ E$ is obtained by undoing the twist in $R$ by twisting $H'\circ E$ in the opposite direction.}
    \label{Cob2}
    \end{figure}
\FloatBarrier

\section{Khovanov homology does not distinguish $D_0\circ H$ and $D_{roll}\circ H$}\label{4}

We assume that the diagram $K$ is chosen such that it has a minimal number of crossings. The presence of nugatory crossings complicates the proof of the main result, so we remove them from the onset to streamline the argument.

\begin{theorem}\label{mainresult}
    Khovanov homology fails to distinguish $D_0\circ H$ and $D_\text{roll}\circ H$ viewed as cobordisms $K\sqcup \bar K\to \varnothing$ up to signs. More precisely, the maps
    \[
        Kh(H),\; Kh(S\circ H'\circ E) : Kh(K\sqcup \bar K)\to Kh(K\#\bar K)
    \]
    agree up to signs.
\end{theorem}

\begin{cor}\label{corollary}
    Let $K$ be a knot such that the $1$-handle move $Kh(H):Kh(K\sqcup \bar K)\to Kh(K\#\bar K)$ is surjective. Then Khovanov homology fails to distinguish the trivial slice disk from the roll-spun slice disk viewed as cobordisms $K\#\bar K\to \varnothing$, up to signs.
\end{cor}

In fact, we show that the equality $Kh(H) = \pm Kh(S\circ H'\circ E)$ already holds on the chain level by exploiting the observation that the generators of $CKh(K\sqcup \bar K)$ not mapped to zero under $CKh(S\circ H'\circ E)$ must take a particular form.

Let $a$ be a generator of $CKh(K\sqcup \bar K)$. Write $Kh(E)(a) = \sum_i b_i$ where $b_i$ are generators of $CKh(K'\sqcup \bar K)$. Consider a crossing $c$ in $K\subset K\sqcup \bar K$. Recall that $K$ is an oriented diagram, and without loss of generality assume that the crossing $c$ is arranged as in Figure \ref{rigidity0}. The over-and-under crossing data are irrelevant to the argument, and thus we suppress it in the figures.

\begin{figure}[!htbp]
    \centering
    \includegraphics[width=0.3\textwidth]{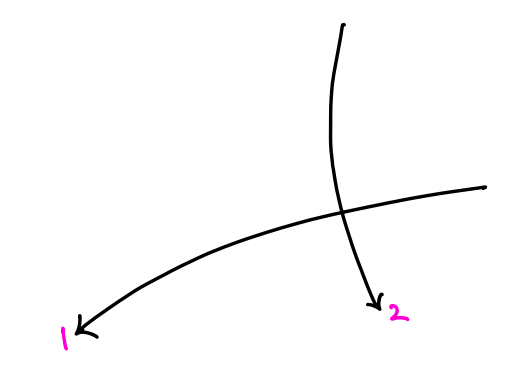}
    \caption{Oriented strands of the knot $K$ around a crossing $c$. The numbering indicates which strand is traversed first as the base of the handle moves along $K$.}
    \label{rigidity0}
\end{figure}
\FloatBarrier

The crossing $c$ could be resolved either by a $0$-smoothing or an $1$-smoothing in $a$. Since the chain map associated to a Reidemeister II move preserves the smoothing data outside of the local picture on which the move is applied to, it follows that the choice of smoothing of $c$ in $a$ is equal to that of $c$ in $b_i$. For ease of argument, let us label the crossings in the local picture around $c$ as in Figure \ref{crossinglabel}.

Since the goal is to compute $CKh(S\circ H'\circ E)(a)$, we may ignore the generators $b_i$ that are mapped to zero under $Kh(S)Kh(H')$. This trick lets us cut down the number of generators we have to consider significantly, in light of the following lemma.

\begin{figure}[!htbp]
    \centering
    \includegraphics[width=0.5\textwidth]{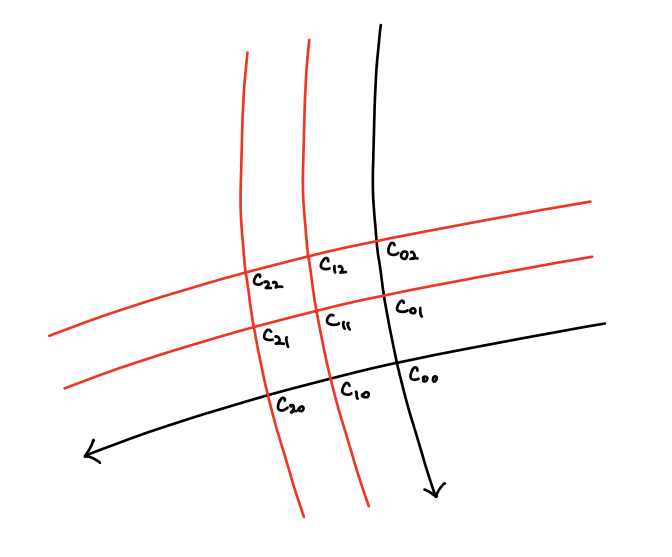}
    \caption{Crossings of $K'\sqcup K$ around $c$. The strands of $K$ are colored black, and strands of the tentacle are colored red.}
    \label{crossinglabel}
\end{figure}
\FloatBarrier

\begin{lemma}\label{rigiditylemma}
    For each smoothing of a crossing $c$ in $K$, there are only two possibilities for smoothings of nearby crossings of $c$ (as in Figure \ref{crossinglabel}) in $K'\sqcup \bar K$ so that the generators are not mapped to $0$ under $CKh(S\circ H')$.
\end{lemma}

\begin{figure}[!htbp]
    \centering
    \includegraphics[width=0.65\textwidth]{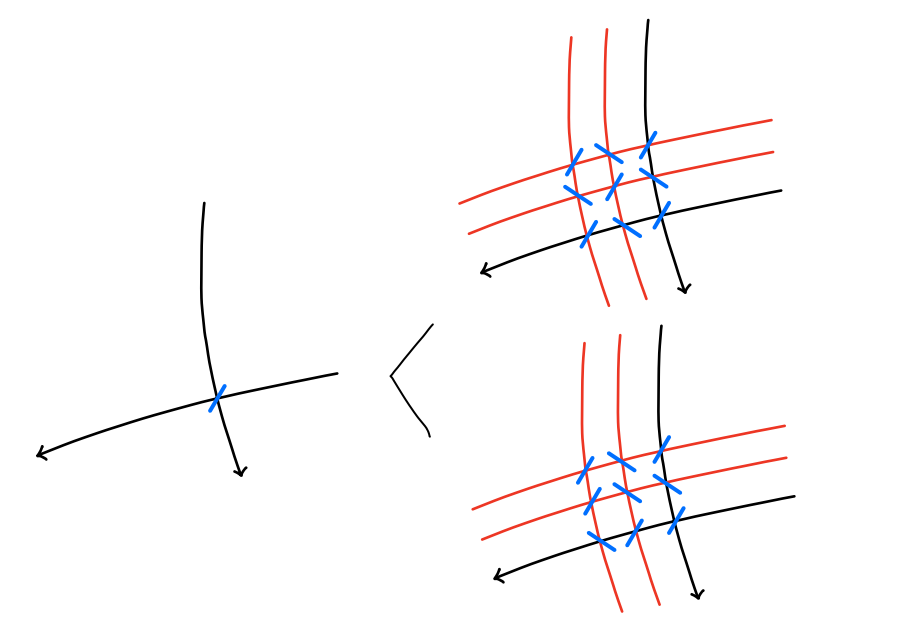}
    \label{smoothing1}
\end{figure}

\begin{figure}[!htbp]
    \centering
    \includegraphics[width=0.65\textwidth]{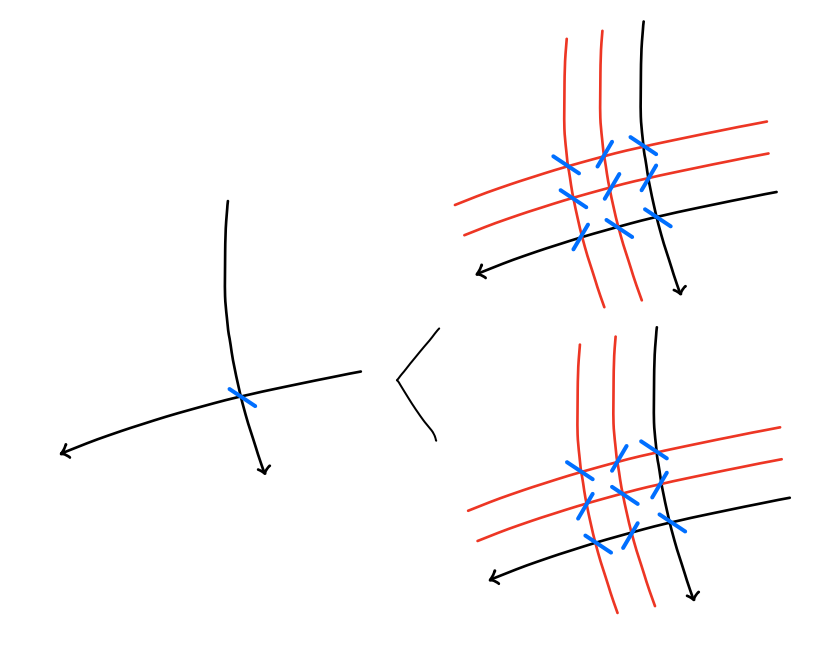}
    \caption{A choice of smoothing of $c$ in $a$ induces two possibilities for smoothings of nearby crossings in $b_i$ not killed by $CKh(S\circ H')$.}
    \label{smoothing2}
\end{figure}
\FloatBarrier

\begin{proof}
    Suppose the crossing $c=c_{00}$ is resolved as in the top left corner of Figure \ref{smoothing1}. Recall that $S$ decomposes into a series of Reidemeister II moves. Inspecting such a Reidemeister move occurring around $c$, we observe that $c_{01}$ must be resolved as in the top left corner of Figure \ref{rigid1} --- otherwise it is killed by the Reidemeister II move, and hence $CKh(S\circ H')$. 

    \begin{figure}[!htbp]
        \centering
        \includegraphics[width=0.65\textwidth]{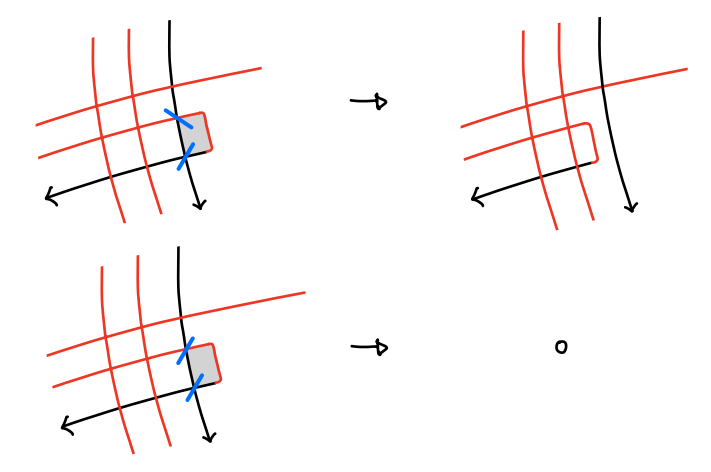}
        \caption{The Reidemeister II move of $S$ on the shaded area induces the identity on the top left smoothing (after planar isotopy), while sending the bottom left smoothing to zero.}
        \label{rigid1}
    \end{figure}
    \FloatBarrier

    Similarly, we inspect a Reidemeister II move of $E$ occuring around $c$. As the tentacle of $E$ first enters the local picture through the horizontal strand of $K$, it creates two crossings $c_{01}$ and $c_{02}$. Recalling the effect of Reidemeister II moves on $CKh$, we note that they must be smoothed as in the right of Figure \ref{rigid2}. But as noticed in the previous argument, the bottom right smoothing is killed by $CKh(S\circ H')$. Hence $c_{02}$ must be resolved as in the top right corner of Figure \ref{rigid2}.

    \begin{figure}[!htbp]
        \centering
        \includegraphics[width=0.65\textwidth]{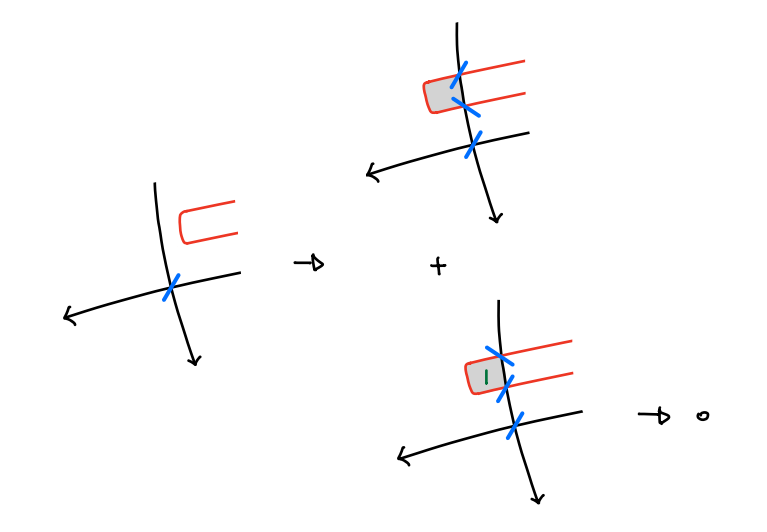}
        \caption{As the tentacle enters the local picture, the Reidemeister II move of $E$ near the shaded area maps the left smoothing to the sum of the two smoothing on the right. The bottom right smoothing is mapped to zero by $CKh(S\circ H')$.}
        \label{rigid2}
    \end{figure}
    \FloatBarrier

    Repeating this process for $c_{12}$ and $c_{22}$, we observe that these crossings are forced to be smoothed as in the right of Figure \ref{rigid3}.

    \begin{figure}[!htbp]
        \centering
        \includegraphics[width=0.65\textwidth]{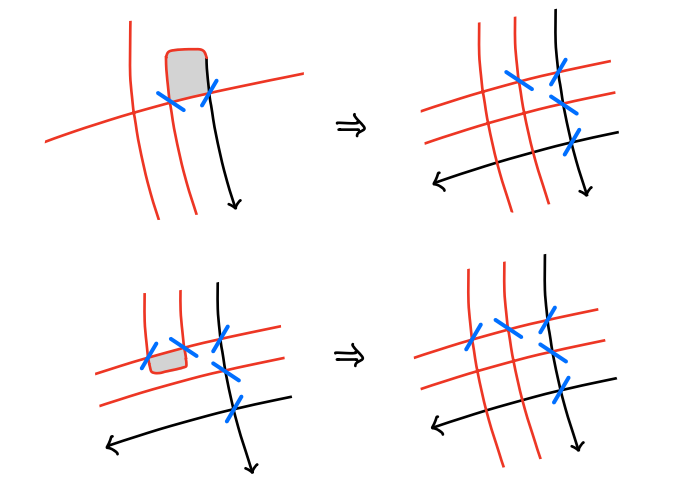}
        \caption{$S$ forces $c_{12}$ to be resolved as in the top right, and $E$ forces $c_{22}$ to be resolved as in the bottom right.}
        \label{rigid3}
    \end{figure}
    \FloatBarrier

    Now there is a choice to make. The remaining crossings are not forced to be resolved in a particular manner. But if we choose a smoothing on a single crossing, say $c_{10}$, the remaining three crossings are forced to be smoothed accordingly. 

    \begin{figure}[!htbp]
        \centering
        \includegraphics[width=0.65\textwidth]{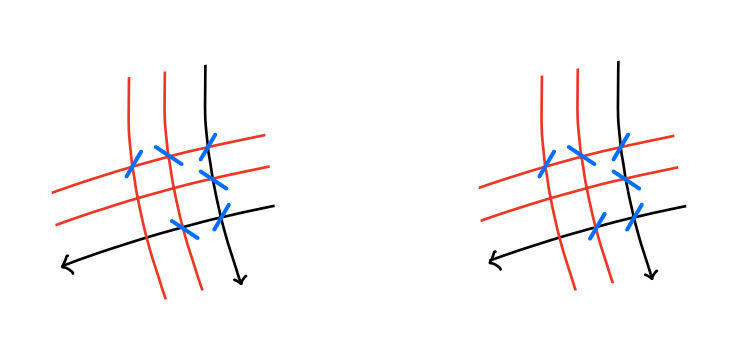}
        \caption{Two choices of smoothings on $c_{10}$.}
        \label{choice}
    \end{figure}
    \FloatBarrier

    Indeed, repeating the process for the remaining crossings $c_{11}$, $c_{2,1}$, and $c_{2,0}$, in that order, we see that the choice of smoothing on $c_{10}$ force these crossings to be smoothed in a particular way. See Figure \ref{rigid4} for the case when $c_{10}$ is assumed to be smoothed as in the left of Figure \ref{choice}.

    \begin{figure}[!htbp]
        \centering
        \includegraphics[width=0.65\textwidth]{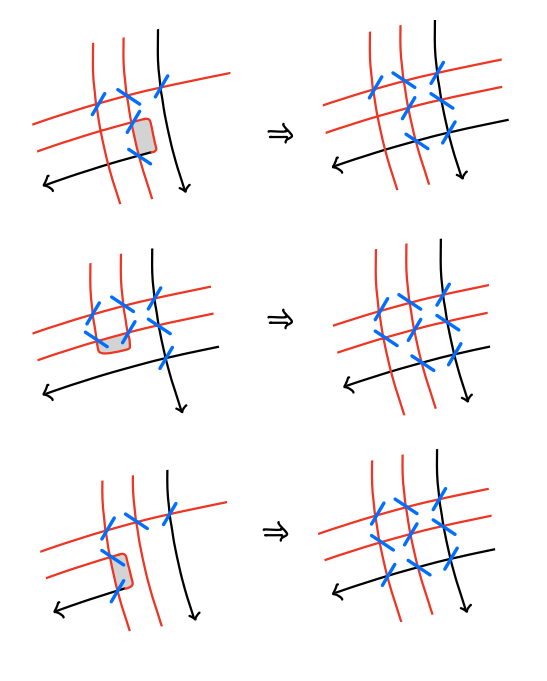}
        \caption{$c_{11}$, $c_{21}$, and $c_{20}$ are forced to be smoothed as in the right.}
        \label{rigid4}
    \end{figure}
    \FloatBarrier

    The case when $c_{00}$ is smoothed in the opposite manner is handled similarly.
\end{proof}

In fact, if one accounts for the smoothing data of \textit{two} adjacent crossings in $a$, we can further cut down the two possibilities to a single case. This will arise as a byproduct of the proof of the following lemmas. Before stating them, we introduce some notations.

\begin{figure}[!htbp]
    \centering
    \includegraphics[width=0.65\textwidth]{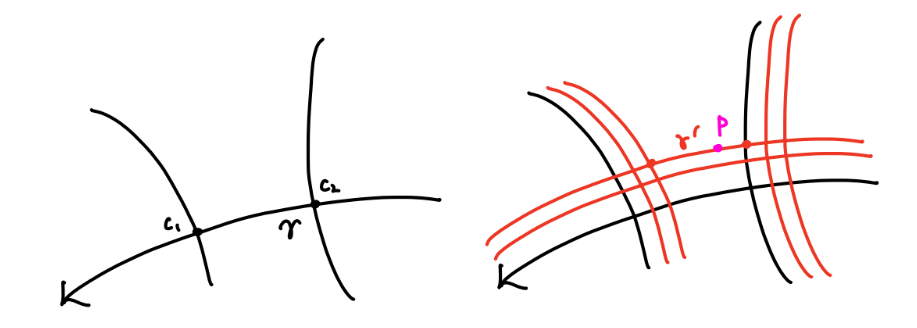}
    \caption{Two adjacent crossings in $b_i$.}
    \label{gamma}
\end{figure}
\FloatBarrier

Let $c_1$ and $c_2$ be two adjacent crossings in $K$. Denote by $\gamma$ the strand in $a$ between $c_1$ and $c_2$. Let $b_i$ be a generator summand of $CKh(E)(a)$. The portion of the tentacle along $\gamma$ is split into two strands, either by another tentacle, or a strand in $K$. Denote by $\gamma'$ the strand further away from $\gamma$. Note that $\gamma'$ is dependent on the summand $b_i$. Choose a point $P$ in the intersection of all $\gamma'$ as $b_i$ ranges over the generator summands of $CKh(E)(a)$. Clearly this intersection is nonempty. 

We need the following definitions for later use.

\begin{defn}\label{defi}
    Let $E= E''\circ E'$, where $E'$ consists of the first $N$ Reidemeister II moves of $E$, and $E''$ the remaining Reidemeister II moves (in the canonical movie decomposition). Let $b'$ be a generator summand of $Kh(E')$, and let $Q$ be a point on the smoothing underlying $b'$ that lies on a strand $\gamma$ of $K$ bounded by two crossings $c_1$ and $c_2$ of $K$. Let $\gamma_1$ and $\gamma_2$ be the strands of $K$ forming $c_1$ and $c_2$ together with $\gamma$, respectively.

    We say that $E'$ has traversed \textbf{over} $Q$ \textbf{through} $c_i$, if a tentacle of $E'$ has crossed $c_i$ along $\gamma_i$. We say that $E'$ has traversed \textbf{along} $Q$ \textbf{through} $c_i$, if a tentacle of $E'$ has crossed $c_i$ along $\gamma$.
\end{defn}

    \begin{figure}[!htbp]
    \centering
    \includegraphics[width=0.65\textwidth]{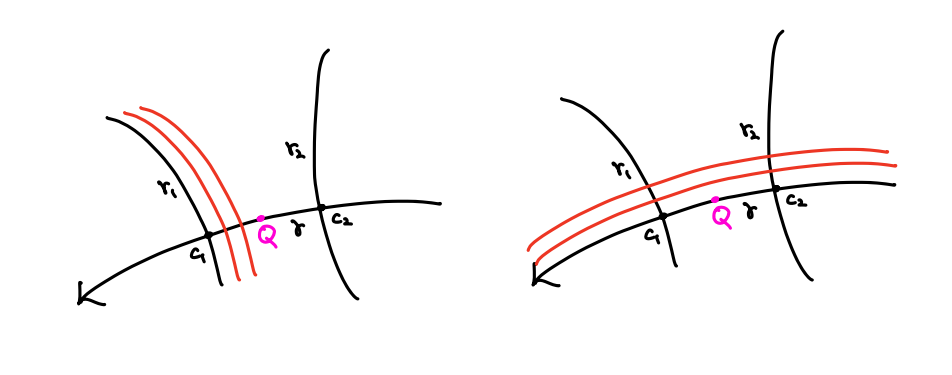}
    \caption{$E'$ traverses over $Q$ through $c_1$ on the left, while $E'$ traverses along $Q$ through $c_1$ and $c_2$.}
    \label{overalong}
    \end{figure}

Note that our minimality assumption on the diagram of $K$ guarantees that $\gamma_1$, $\gamma_2$, and $\gamma$ are all distinct.

Our goal is to show that a generator summand of $CKh(E)(a)$, not killed by $CKh(S\circ H')$, must have the same label at $P$ as the component of $a$ containing $\gamma$. This is shown by repeatedly applying the following technical lemmas.

\begin{lemma}\label{kill1}
    Decompose $E=E''\circ E'$. Let $b'$ be a generator summand of $CKh(E')(a)$. If $b'$ has a component labeled $1$ that lies between a strand of $K$ and the lower strand of a tentacle, then $CKh(S\circ H'\circ E'')(b')=0$.
\end{lemma}

\begin{proof}
    We handle one case. Suppose $b'$ has a component $C$ labeled $1$ configured as in the upper left corner of Figure \ref{kill1lemma}. Note that $C$ does not leave the local picture, so we may assume that it stays labeled $1$ until the tentacle returns to the local picture. After a sequence of Reidemeister II moves in $E''$, the tentacle leaves the local picture, splitting $C$ into two components $C_1$ and $C_2$. 
    
    The generator summands of $CKh(E'')(b')$ can be grouped into two categories depending on how $C_1$ and $C_2$ are labeled. On one hand, we have summands whose $C_1$ and $C_2$ components are labeled $1$ and $x$ respectively (lower diagram in the upper right corner of Figure \ref{kill1lemma}), and on the other hand we have summands labeled $x$ and $1$ respectively (upper diagram in the upper right corner of Figure \ref{kill1lemma}). Let $b_1'$, $b_2'$ be generator summands of $CKh(E'')(b')$ in the first and second categories respectively.

    Since $C_1$ and $C_2$ does not leave the local picture of Figure \ref{kill1lemma}, in computing the effect of $CKh(S\circ H')$ on $b_1'$ and $b_2'$, we need only study the effect of a Reidemeister II move of $S$ occurring within the local picture. This is computed in the remainder of Figure \ref{kill1lemma}, revealing that both $CKh(S\circ H')(b_1')$ and $CKh(S\circ H')(b_2')$ are zero.

    The remaining cases (depending on the orientation and ordering of the strands in the local pictures) are handled analogously.
\end{proof}

    \begin{figure}[!htbp]
    \centering
    \includegraphics[width=0.85\textwidth]{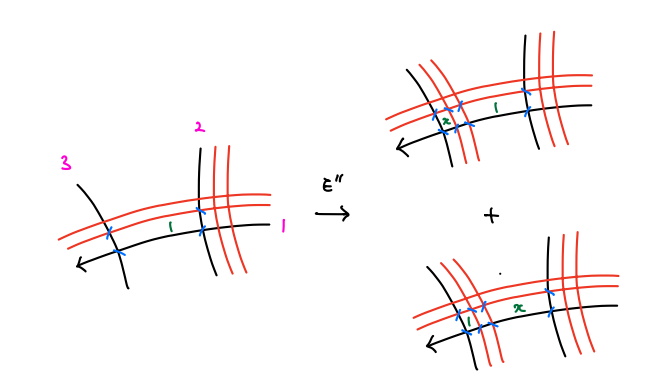}
    \end{figure}
    \FloatBarrier

    \begin{figure}[!htbp]
    \centering
    \includegraphics[width=0.85\textwidth]{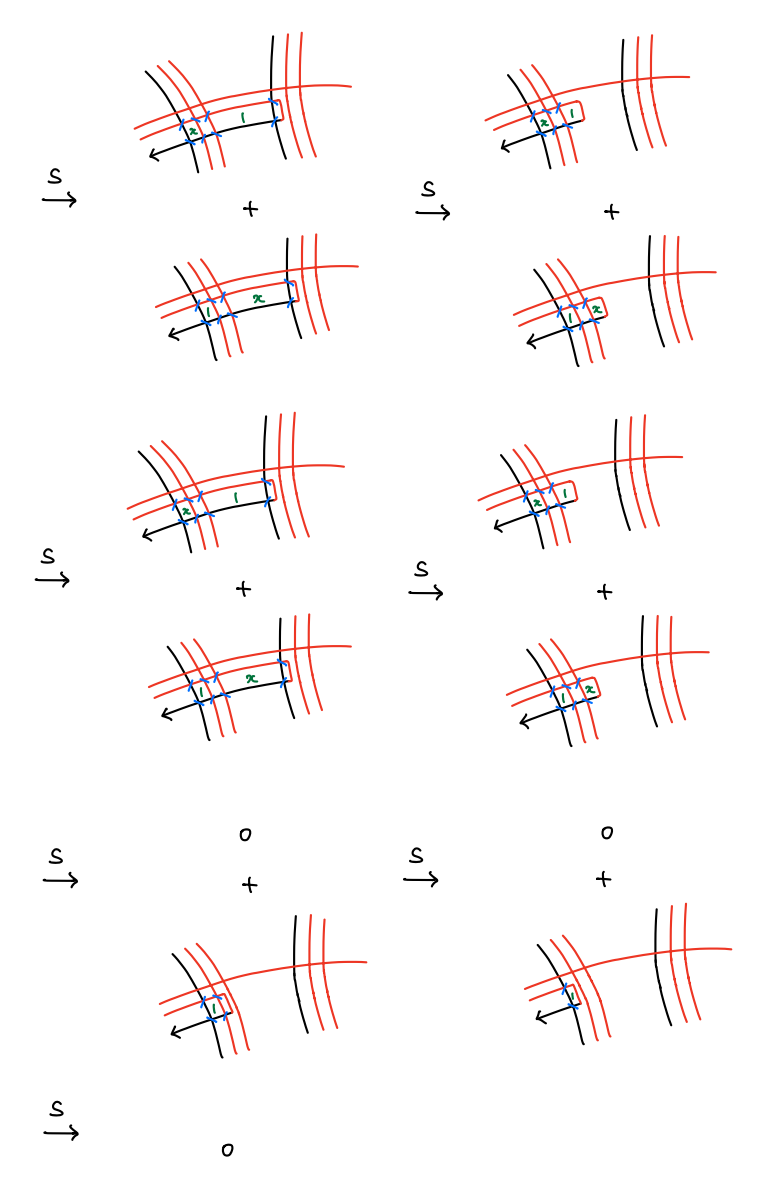}
    \caption{Verification of Lemma \ref{kill1} in the case when the strands are oriented and ordered as in the top left corner.}
    \label{kill1lemma}
    \end{figure}
    \FloatBarrier

\begin{lemma}\label{birth1}
     Decompose $E=E''\circ E'$, and let $b'$ be a generator summand of $Kh(E')(a)$. Then a component of $b'$ completely encased between the two strands of a tentacle is labeled $1$. 
\end{lemma}

\begin{proof}
    We handle one case. Clearly we need only consider local pictures of $b'$ as in the top left corner of Figure \ref{birth1lemma}, and show that any component $C$ arising from further Reidemeister II moves that is completely encased within the two strands of a tentacle must be labeled $1$. This is shown in Figures \ref{birth1lemma}.

    The remaining cases (depending on the orientation and ordering of the strands in the local pictures) are handled analogously.
\end{proof}

    \begin{figure}[!htbp]
    \centering
    \includegraphics[width=0.5\textwidth]{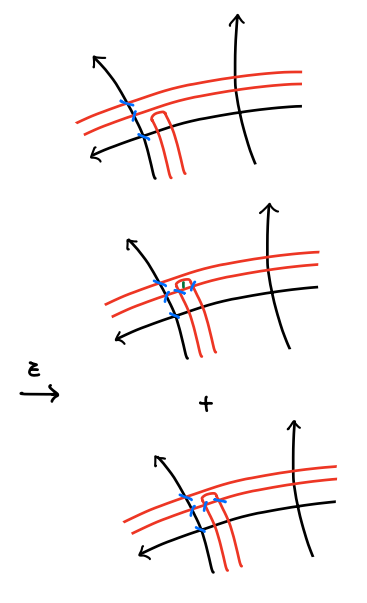}
    \label{birth1lemma}
    \end{figure}
    \FloatBarrier

    \begin{figure}[!htbp]
    \centering
    \includegraphics[width=0.85\textwidth]{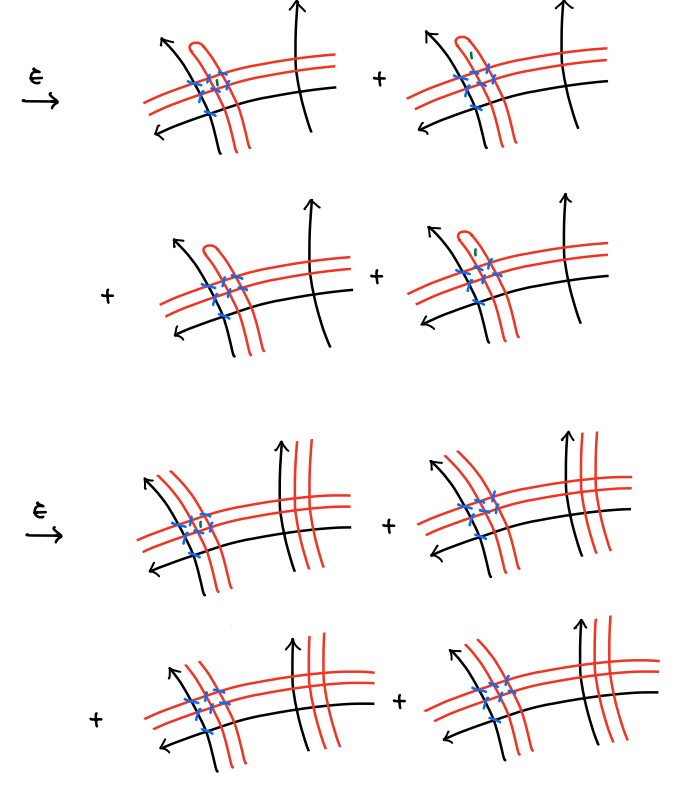}
    \caption{Verification of Lemma \ref{birth1} in the case when the strands are oriented and ordered as in the top left corner.}
    \label{birth1lemma}
    \end{figure}
    \FloatBarrier

Iterating the above two principles yields the following technical results. We defer the proofs to the appendix.

\begin{prop}\label{localElemma}
    Suppose that the component of $a$ containing $\gamma$ is labeled $1$. Let $b_i$ be a generator summand of $Kh(E)(a)$ not mapped to zero under $Kh(S)Kh(H')$. Then the component of $b_i$ containing $P$ is also labeled $1$.
\end{prop}

\begin{lemma}\label{farEmovelemma1}
    Let $E = E''\circ E'$, $b'$, and $Q$ be in Definition \ref{defi}. Suppose $E'$ has not traversed along $Q$. Decompose $E''$ into $E_2''\circ E_1''$, where $E_1''$ is maximal among summands of $E''$ consisting only of moves that do not traverse along $Q$. Suppose the component of $b'$ containing $Q$ is labeled $1$. Then the component of $Kh(E_1'')(b')$ containing $Q$ is also labeled $1$.
\end{lemma}

We also need analogous lemmas for $x$ in place of $1$. The proofs are completely analogous.

\begin{prop}\label{localElemmax}
    Suppose that the component of $a$ containing $\gamma$ is labelled $x$. Let $b'$ be a generator summand of $Kh(E)(a)$ not mapped to zero under $Kh(S)Kh(H')$. Then the component of $b'$ containing $P$ is also labelled $x$.
\end{prop}

\begin{lemma}\label{farEmovelemmax}
    Let $E = E''\circ E'$, $b'$, and $Q$ be as in Lemma \ref{farEmovelemma1}. Suppose $E'$ has not traversed along $Q$. Decompose $E''$ into $E_2''\circ E_1''$, where $E_1''$ is maximal among summands consisting only of moves that do not traverse along $Q$. Suppose the component of $b'$ containing $Q$ is labeled $x$. Then the component of $Kh(E_1'')(b')$ containing is also labeled $x$.
\end{lemma}

Propositions \ref{localElemma} and \ref{localElemmax} described the local labels when the strand in question is bounded by two crossings. To conclude the proof of the main result, we need to handle two exceptional cases. Namely, when the tentacle traverses the first crossing, and when the tentacle traverses the last crossing. These are taken care of by the following lemmas.

\begin{lemma}\label{beginlemma}
    Let $b'$ be a generator summand of $Kh(E)(a)$ such that $Kh(S\circ H')(b')\neq 0$. Let $Q$ be a point of $K$ around the attaching region for the 1-handle $H$. Then the labels of $a$ and $b'$ at $Q$ are the same.
\end{lemma}

\begin{lemma}\label{endlemma}
    Let $P$, $Q$ be points of $K$, $\bar K$ near the attaching region of the $1$-handle $H$. Let $a\in CKh(K\sqcup \bar K)$ be a generator, and suppose that the component of $a$ containing $P$ is labeled $p$, and the component containing $Q$ is labeled $q$. Let $b'$ be a generator summand of $CKh(H'\circ E)(a)$ not mapped to zero under $CKh(S)$. Then the component of $b'$ containing the $1$-handle $H'$ is labeled $m(p\otimes q)$.
\end{lemma}

Note that the component of $CKh(H)(a)$ containing the $1$-handle $H$ is labeled $m(p\otimes q)$ as well. 

Now we prove the main result. The argument will proceed by mimicking the proofs of Propositions \ref{localElemma} and \ref{localElemmax} and Lemmas \ref{farEmovelemma1} and \ref{farEmovelemmax}, but for the cobordism $S$ in place of $E$.

\begin{proof}[Proof of Theorem \ref{mainresult}]
    Let $b'$ be a generator summand of $Kh(H')Kh(E)(a)$ not mapped to zero under $Kh(S)$. The proofs of Propositions \ref{localElemma}, \ref{localElemmax} and Lemmas \ref{beginlemma}, \ref{endlemma} imply that the smoothing data of two adjacent crossings, and the label of the strand bounded by the crossings in $a$, completely determine the generator summand. In particular, such a generator summand $b'$ exists, and is unique.

    By Lemmas \ref{kill1} and \ref{birth1}, we know that every component of $b'$ lying strictly between a strand of $K$ and the lower strand of the tentacle must be labeled $x$, and that every component bounded within the tentacle must be labeled $1$. We claim that up to signs, this forces $S$ to act by simply erasing the strand along $K$ and the bottom strand of the tentacle along $K$ while keeping the labels on the upper strand fixed. Since the intersections of the upper strands of tentacles are resolved and labeled in the same way as the corresponding crossings on $K$, it follows from the claim that $CKh(S\circ H'\circ E)(a) = CKh(S)(b') = \pm CKh(H)(a)$. The sign comes from the $-$ sign in the Reidemeister II move induced chain map on $CKh$. More precisely, the sign counts the mod $2$ number of components of $b'$ lying strictly between a strand of $K$ and a lower strand of the tentacle.

    Hence we are reduced to proving the claim. Figures \ref{localElemma1}-\ref{localElemma4'} and \ref{localElemmax1}-\ref{localElemmax4'} force the local picture of $b'$ around a strand $\gamma$ of $K$ to fit into one of the eight cases described by Figures \ref{localSlemma0-1}-\ref{localSlemma0-2}, provided that the strands are oriented as described (left vertical strand upwards, right vertical strand downwards). It suffices to show that the moves of $S$ keep the label at $P$ fixed. The verification of the claim for the first case, for moves of $S$ occurring in the local picture, is detailed in Figure \ref{localSlemma1'-1'}. The verification for the remaining cases (for each choice of orientations and resolutions) are completely similar.

    For moves of $S$ occurring outside the local picture, there are eight cases to consider for each choice of orientation and resolution, as described in Figures \ref{farSlemma0-1}-\ref{farSlemma0'}. The verification of the claim for the first case is detailed in Figure \ref{farSlemma1}. Again, the verification for the remaining cases is completely similar.
\end{proof}

In light of Corollary \ref{corollary}, we end with the following interesting question. For which knots $K$ is the $1$-handle map $Kh(K\sqcup \bar K)\to Kh(K\#\bar K)$ surjective?

\section*{Appendix}

In this appendix we collect the proofs of the technical results in Section \ref{4}.

\subsection{Proof of Propsition \ref{localElemma}.}
    Suppose without loss of generality that $\gamma$ is configured as in Figure \ref{localElemma0}. Decompose
    \[
    E =  E_3\circ \rho_3\circ E_2\circ\rho_2\circ E_1 \circ \rho_1\circ E_0,
    \]
    where $\rho_i$ indicates the sequence of Reidemeiester II moves occurring in the local picture of $\gamma$, along the $i$-th strand of $K$, and $E_i$ indicates the sequence of interloping Reidemeister II moves of $E$ occurring outside the local picture.

    We will later show in Lemma \ref{farEmovelemma1} that a generator summand of $CKh(E_i)(b')$ (for $i>0$) not killed by $CKh(S\circ H')$ is labeled $1$ at $P$, provided $b'$ is labeled $1$ at $P$. In light of this, we need only consider the action of $\rho_i$ on the local picture.
    
    There are four ways to resolve the crossings $c_1$ and $c_2$. In each case, Lemma \ref{kill1} forces the generator summands of $CKh(\rho_1\circ E_0)(a)$ to take the form in Figure \ref{localElemma1-4}.

    Now we consider the effect of $CKh(\rho_2)$ on $CKh(E_1\circ \rho_1\circ E_0)(a)$. To each of the four cases in Figure \ref{localElemma1-4}, there are two possible cases, depending on which of the two vertical strands are traversed first. The resulting computations are gathered in Figures \ref{localElemma1}-\ref{localElemma4'}. Note that these are the only orientations of the vertical strands we have to consider, since the tentacles associated to any other orientations would not interact with the strand $\gamma$. Also note that $P$ is labeled $1$ in every generator summand not killed by $CKh(S\circ H')$. 

    The effect of $CKh(\rho_3)$ on $CKh(E_2\circ\rho_2\circ E_1\circ \rho_1\circ E_0)(a)$ can then be computed as follows. We apply the case when the left vertical strand is traversed first to the case when the right vertical strand is traversed first, by looking at the local picture of a strand bounded by the outer strand of a tentacle along the right vertical strand of $K$, $\gamma$, and the left vertical strand of $K$. This handles the case when the right vertical strand is traversed before the left vertical strand. 

    Vice versa, we can handle the remaining case when the left vertical strand is traversed before the right vertical strand. In fact, the two resulting local pictures are the same. i.e., there is no difference in the ordering of the strands when looking locally along $\gamma$.\qed

    \begin{figure}[!htbp]
    \centering
    \includegraphics[width=0.4\textwidth]{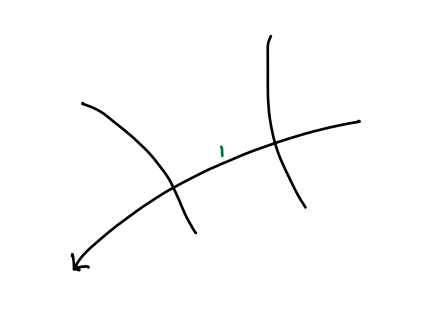}
    \caption{local diagram of $K$ with strand $\gamma$ bounded by crossings $c_1$, $c_2$ labeled $1$.}
    \label{localElemma0}
    \end{figure}
    \FloatBarrier

    \begin{figure}[!htbp]
    \centering
    \includegraphics[width=0.65\textwidth]{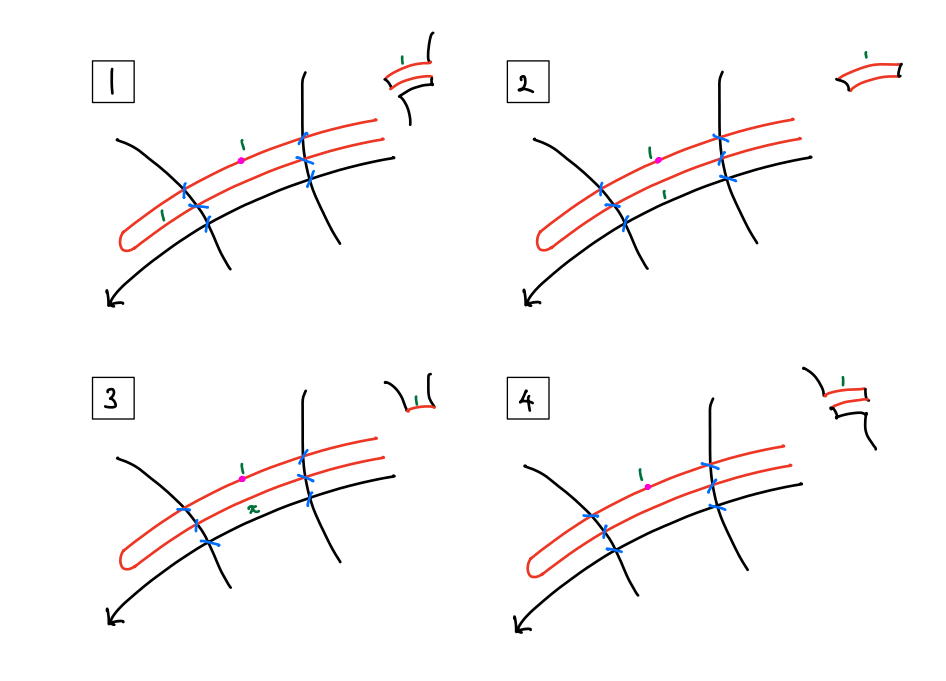}
    \caption{There are four ways to resolve the crossings $c_1$, $c_2$ in Figure \ref{localElemma0}. Rigidity forces the generator summands of $Kh(E')(a)$ that are not killed by $Kh(S\circ H'\circ E'')$ to be smoothed as above. Here we are using the notations of Lemma \ref{farEmovelemma1}. In the upper right corner of each local picture we describe how the component containing $P$ looks like.}
    \label{localElemma1-4}
    \end{figure}
    \FloatBarrier

    \begin{figure}[!htbp]
    \centering
    \includegraphics[width=0.65\textwidth]{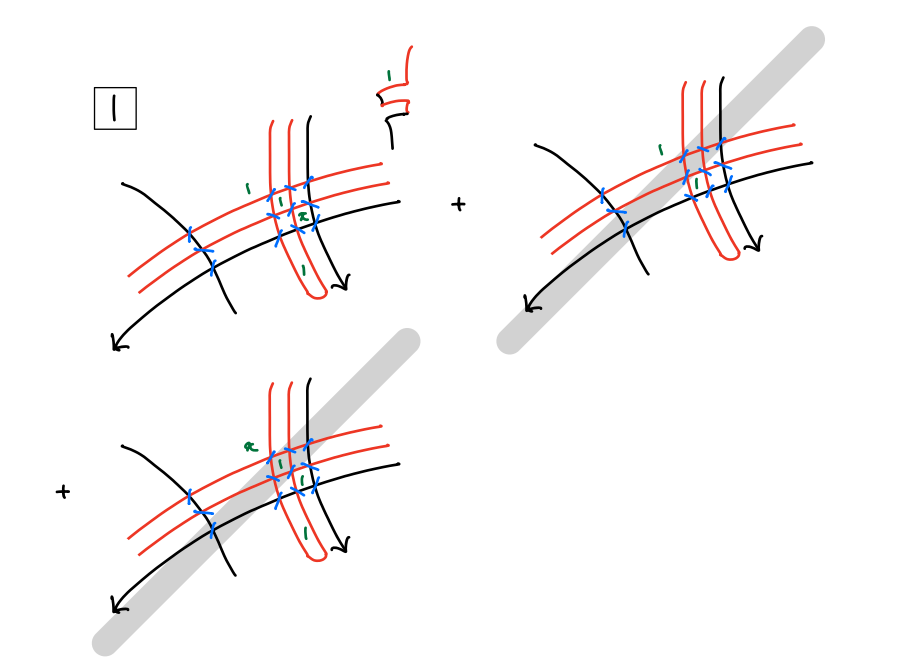}
    \caption{Case 1 of Figure \ref{localElemma1-4} when the tentacle enters the local picture from the right vertical strand first. Note that we need only consider the case when the right vertical strand is oriented downwards. Else, the tentacle does not interact with the component of interest within the local picture, and thus there is nothing to prove. By Lemma \ref{kill1}, the summands that have been struck out (in gray) are killed by $Kh(S\circ H'\circ E'')$.} 
    \label{localElemma1}
    \end{figure}
    \FloatBarrier

    \begin{figure}[!htbp]
    \centering
    \includegraphics[width=0.65\textwidth]{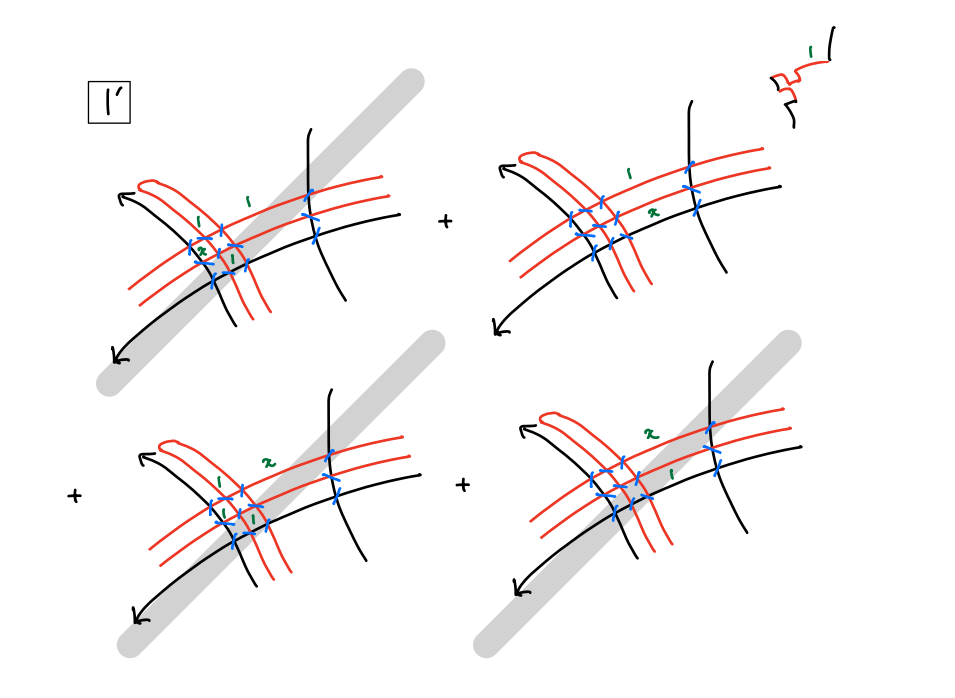}
    \caption{Case 1 of Figure \ref{localElemma1-4} when the tentacle enters the local picture from the left vertical strand first. Note that we need only consider the case when the left vertical strand is oriented upwards.}
    \label{localElemma1'}
    \end{figure}
    \FloatBarrier

    \begin{figure}[!htbp]
    \centering
    \includegraphics[width=0.65\textwidth]{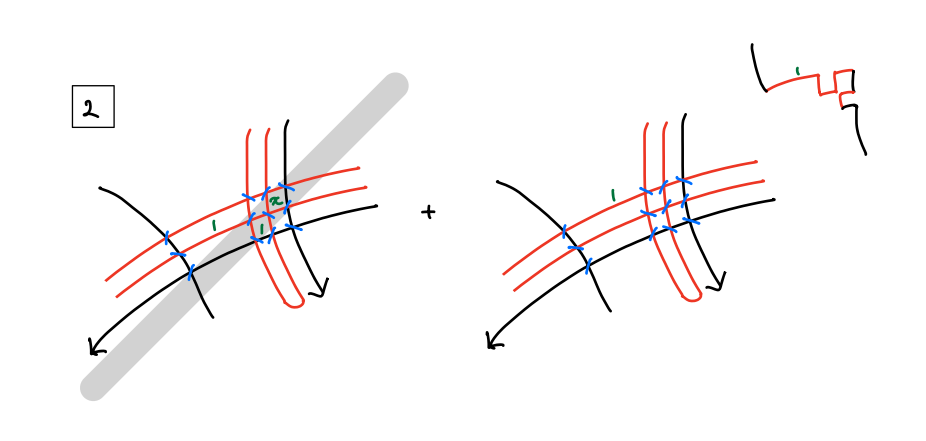}
    \end{figure}
    \FloatBarrier

    \begin{figure}[!htbp]
    \centering
    \includegraphics[width=0.65\textwidth]{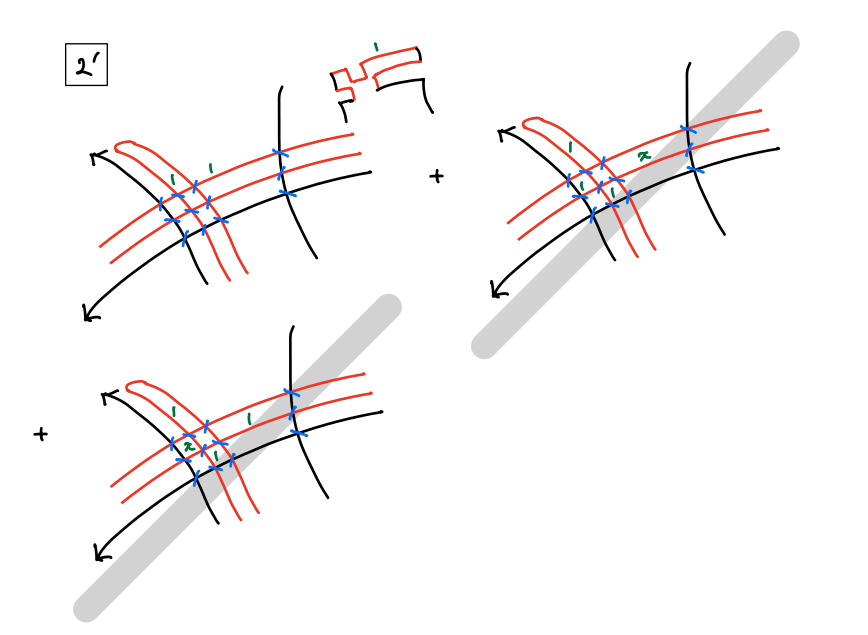}
    \caption{Case 2 of Figure \ref{localElemma1-4}.}
    \label{localElemma2'}
    \end{figure}
    \FloatBarrier

    \begin{figure}[!htbp]
    \centering
    \includegraphics[width=0.65\textwidth]{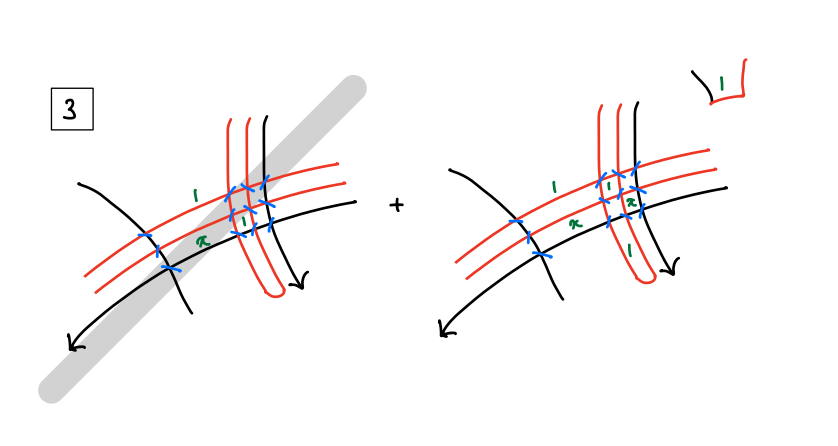}
    \label{localElemma3}
    \end{figure}
    \FloatBarrier

    \begin{figure}[!htbp]
    \centering
    \includegraphics[width=0.65\textwidth]{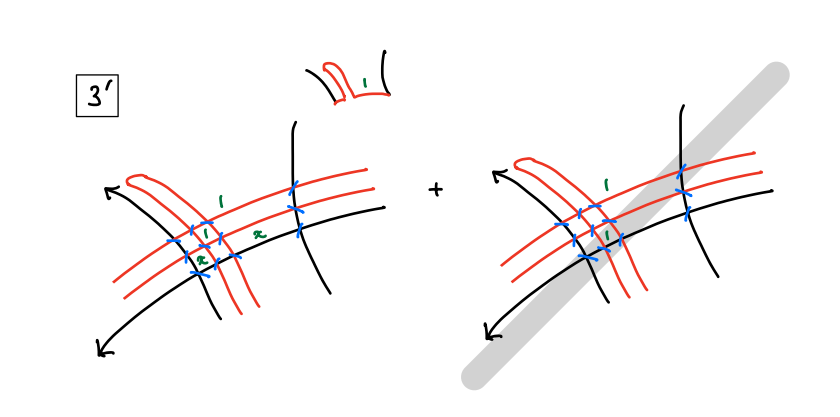}
    \caption{Case 3 of Figure \ref{localElemma1-4}.}
    \label{localElemma3'}
    \end{figure}
    \FloatBarrier

    \begin{figure}[!htbp]
    \centering
    \includegraphics[width=0.65\textwidth]{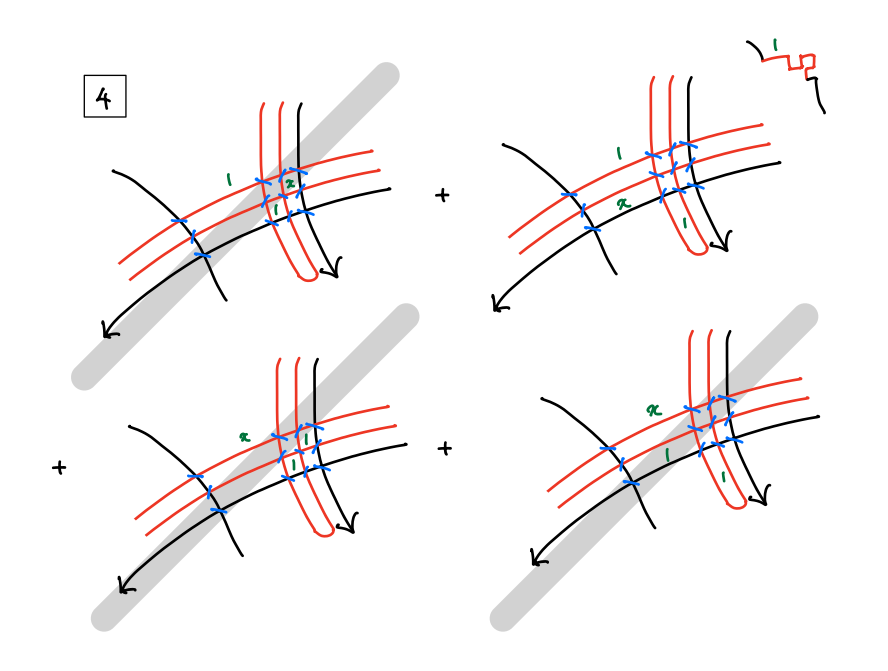}
    \label{localElemma4}
    \end{figure}
    \FloatBarrier

    \begin{figure}[!htbp]
    \centering
    \includegraphics[width=0.65\textwidth]{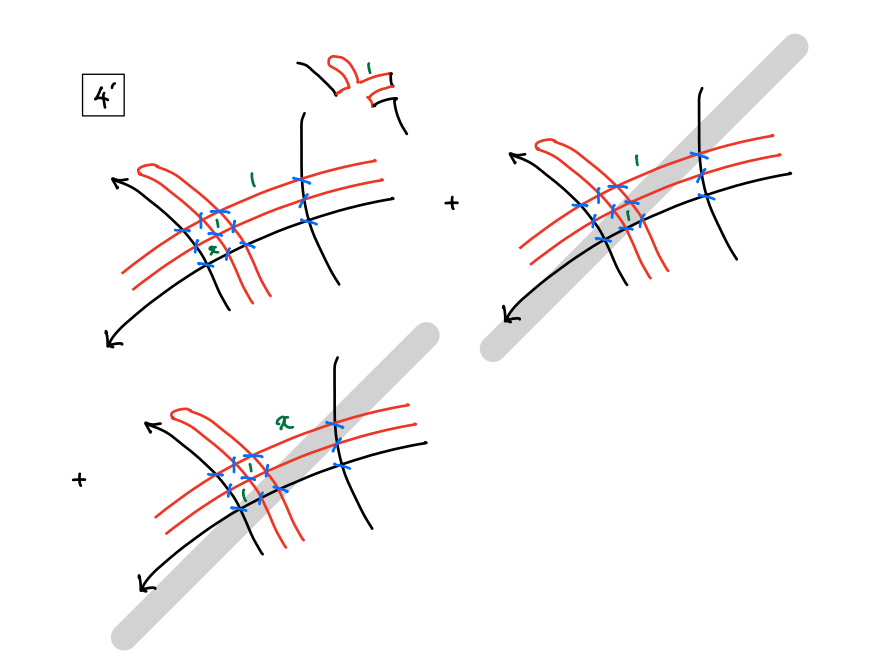}
    \caption{Case 4 of Figure \ref{localElemma1-4}.}
    \label{localElemma4'}
    \end{figure}
    \FloatBarrier

\subsection{Proof of Proposition \ref{localElemmax}}

    The proof is completely analogous to that of Proposition \ref{localElemma}, using Lemma \ref{farEmovelemmax} in place of Lemma \ref{farEmovelemma1}. The relevant computations are detailed in Figures \ref{localElemmax0}-\ref{localElemmax4'}.\qed
    
    \begin{figure}[!htbp]
    \centering
    \includegraphics[width=0.4\textwidth]{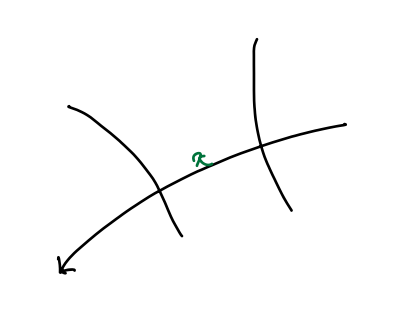}
    \caption{local diagram of $K$ with strand $\gamma$ bounded by crossings $c_1$, $c_2$ labeled $x$.}
    \label{localElemmax0}
    \end{figure}
    \FloatBarrier

    \begin{figure}[!htbp]
    \centering
    \includegraphics[width=0.65\textwidth]{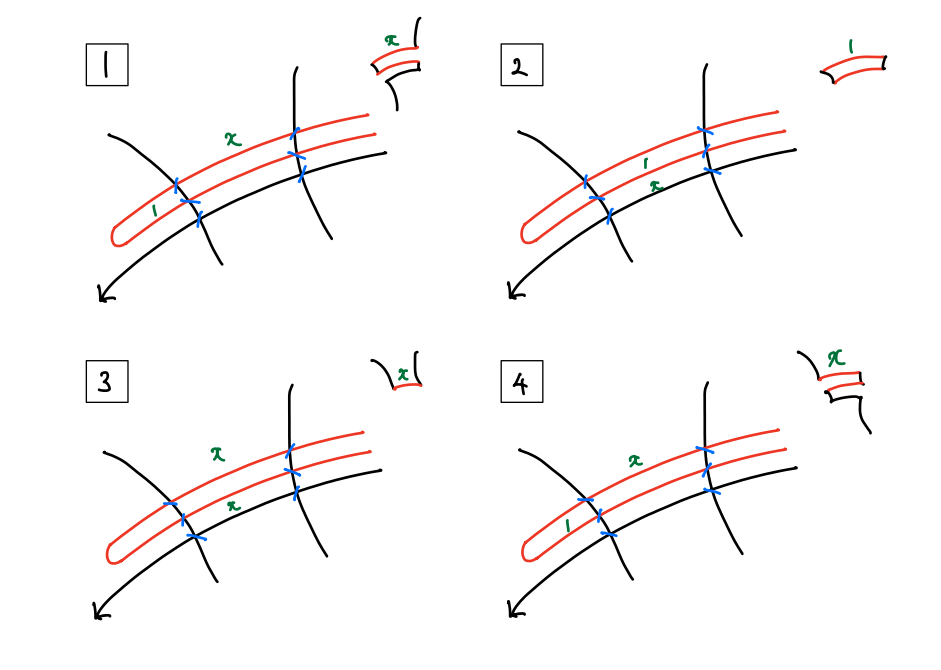}
    \caption{There are four ways to resolve the crossings $c_1$, $c_2$ in Figure \ref{localElemmax0}. Rigidity forces the generator summands of $Kh(E')(a)$ that are not killed by $Kh(S\circ H'\circ E'')$ to be smoothed as above. Here we are using the notations of Lemma \ref{farEmovelemmax}. In the upper right corner of each local picture we describe how the component containing $P$ looks like.}
    \label{localElemmax1-4}
    \end{figure}
    \FloatBarrier

    \begin{figure}[!htbp]
    \centering
    \includegraphics[width=0.65\textwidth]{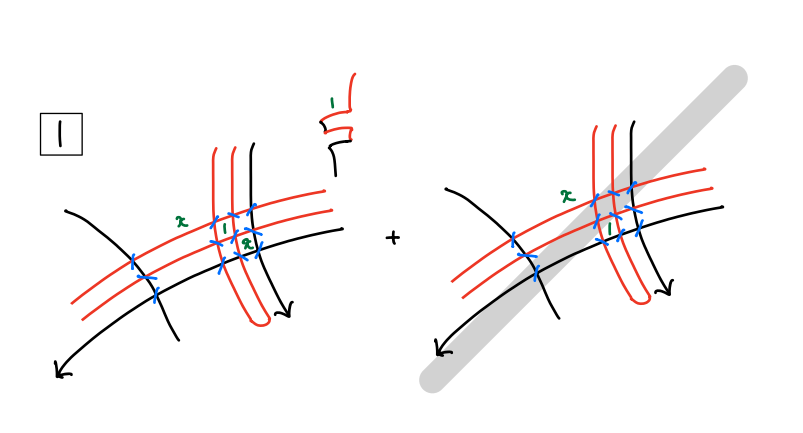}
    \caption{Case 1 of Figure \ref{localElemmax1-4} when the tentacle enters the local picture from the right vertical strand first. Note that we need only consider the case when the right vertical strand is oriented downwards.} 
    \label{localElemmax1}
    \end{figure}
    \FloatBarrier

    \begin{figure}[!htbp]
    \centering
    \includegraphics[width=0.65\textwidth]{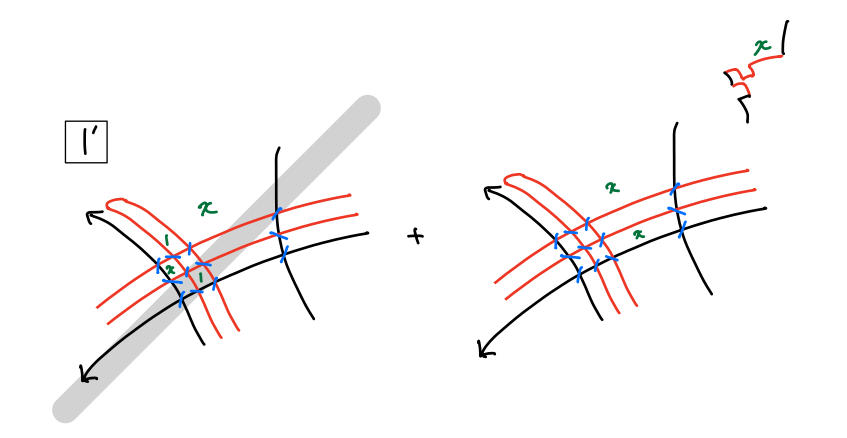}
    \caption{Case 1 of Figure \ref{localElemmax1-4} when the tentacle enters the local picture from the left vertical strand first. Note that we need only consider the case when the left vertical strand is oriented upwards.}
    \label{localElemmax1'}
    \end{figure}
    \FloatBarrier

    \begin{figure}[!htbp]
    \centering
    \includegraphics[width=0.65\textwidth]{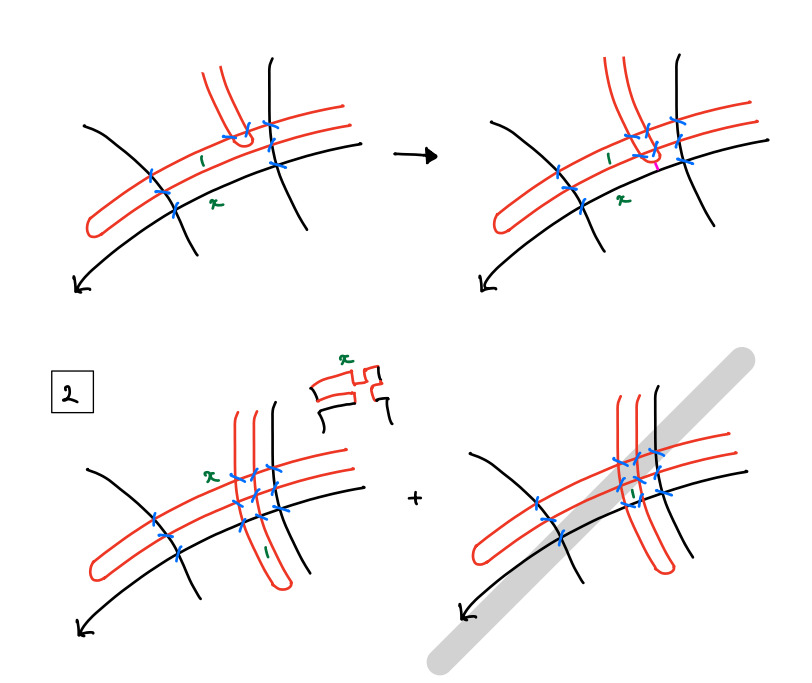}
    \label{localElemmax2}
    \end{figure}
    \FloatBarrier

    \begin{figure}[!htbp]
    \centering
    \includegraphics[width=0.65\textwidth]{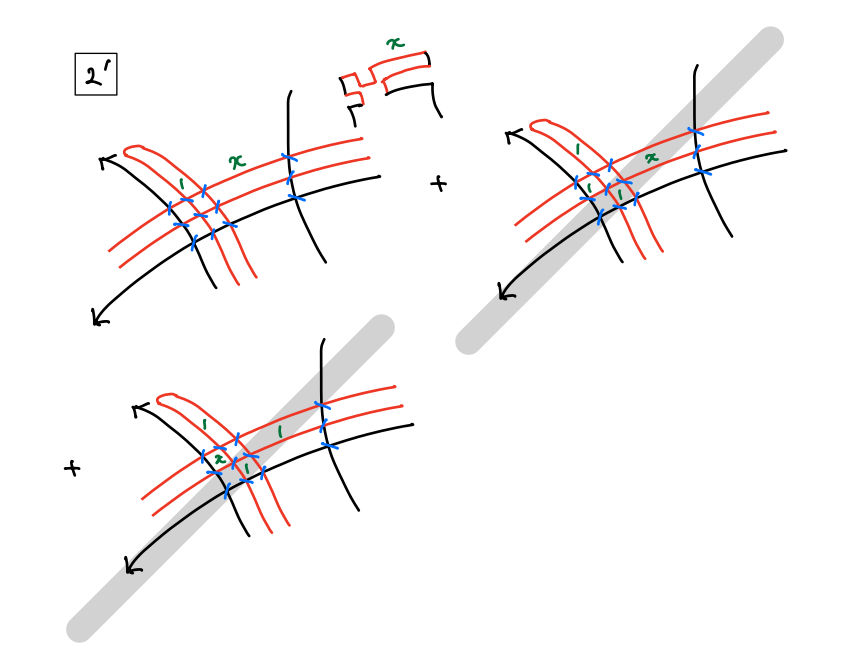}
    \caption{Case 2 of Figure \ref{localElemmax1-4}.}
    \label{localElemmax2'}
    \end{figure}
    \FloatBarrier

    \begin{figure}[!htbp]
    \centering
    \includegraphics[width=0.65\textwidth]{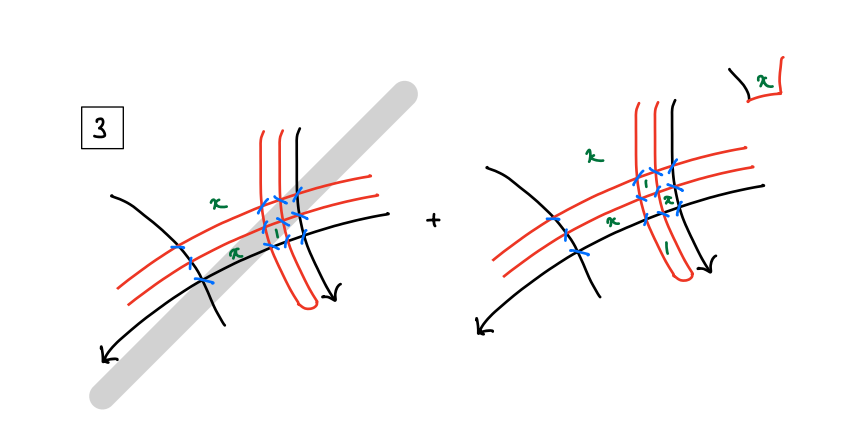}
    \label{localElemmax3}
    \end{figure}
    \FloatBarrier

    \begin{figure}[!htbp]
    \centering
    \includegraphics[width=0.65\textwidth]{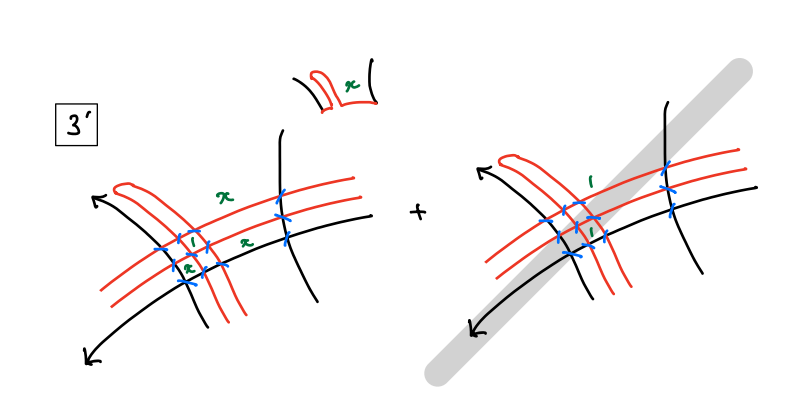}
    \caption{Case 3 of Figure \ref{localElemmax1-4}.}
    \label{localElemmax3'}
    \end{figure}
    \FloatBarrier

    \begin{figure}[!htbp]
    \centering
    \includegraphics[width=0.65\textwidth]{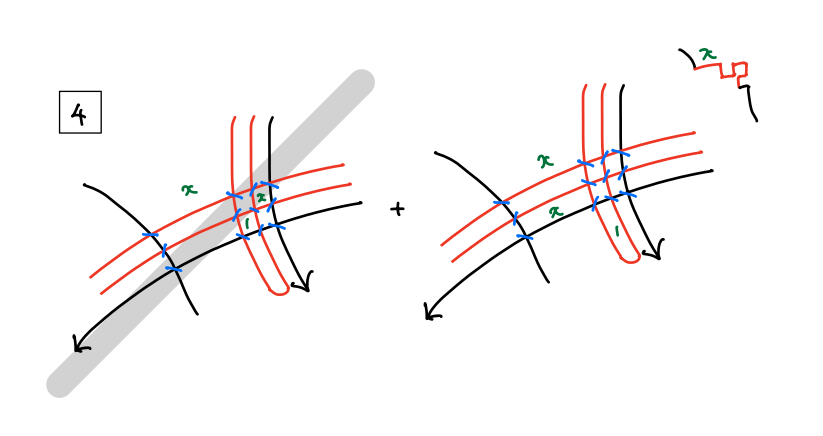}
    \label{localElemmax4}
    \end{figure}
    \FloatBarrier

    \begin{figure}[!htbp]
    \centering
    \includegraphics[width=0.65\textwidth]{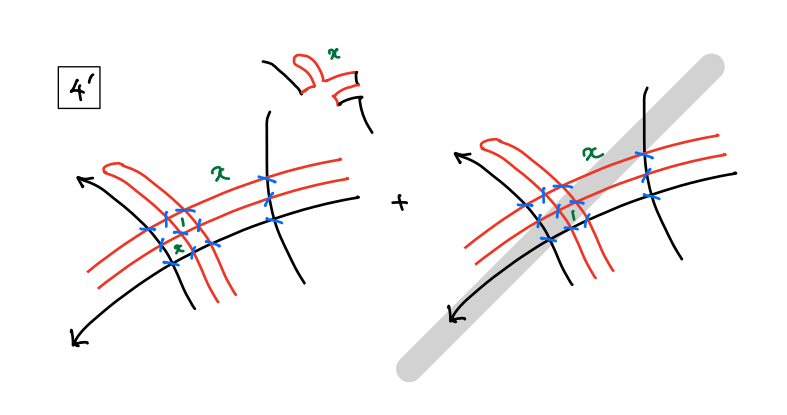}
    \caption{Case 4 of Figure \ref{localElemmax1-4}.}
    \label{localElemmax4'}
    \end{figure}
    \FloatBarrier

\subsection{Proof of Lemma \ref{farEmovelemma1}}

    By induction, it suffices to show that a strand $\gamma$ labeled $1$ in $b'$ that enters a local picture of a crossing stays labeled $1$, after $E''_1$ traverses over $Q$ through the strand other than $\gamma$.
    
    When $\gamma$ is a strand of $K$, there are six possible cases depending on how the crossing is resolved, and how the other strand is oriented. These are described in Figures \ref{farElemma0}-\ref{farElemma2}. The configurations not belonging to these six cases can be ignored, since the tentacle would not interact with the component containing $\gamma$ in the local picture. By rotating these figures and applying the appropriate computations in Figures \ref{localElemma1}-\ref{localElemma4'}, we can conclude that the tentacle along the horizontal strand of $K$ fixes the label at $\gamma$ as well.

    When $\gamma$ is a strand of the tentacle, there are three possible cases, depending on how the crossing is smoothed, and whether $\gamma$ belongs to the upper or the lower strand of the tentacle. These are handled in Figures \ref{farElemmaR0}-\ref{farElemmaR3}. Again, configurations not described in these figures can be ignored.\qed

    \begin{figure}[!htbp]
    \centering
    \includegraphics[width=0.65\textwidth]{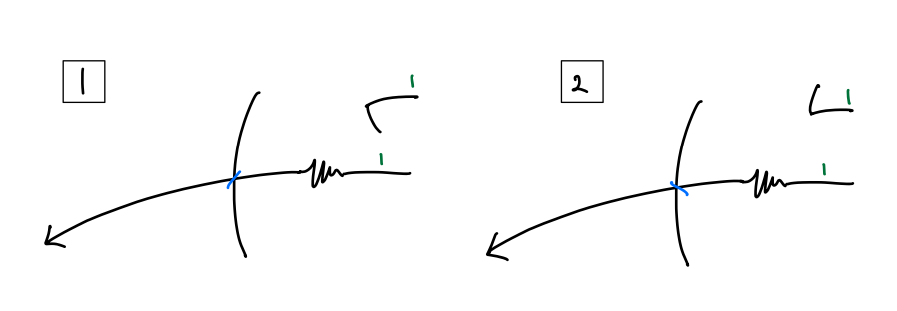}
    \caption{The strand labeled $1$ escapes the local picture of Proposition \ref{localElemma} along a strand of $K$ (drawn black), and encounters a crossing far away from the local picture. There are two ways this crossing can be resolved.}
    \label{farElemma0}
    \end{figure}
    \FloatBarrier

    \begin{figure}[!htbp]
    \centering
    \includegraphics[width=0.65\textwidth]{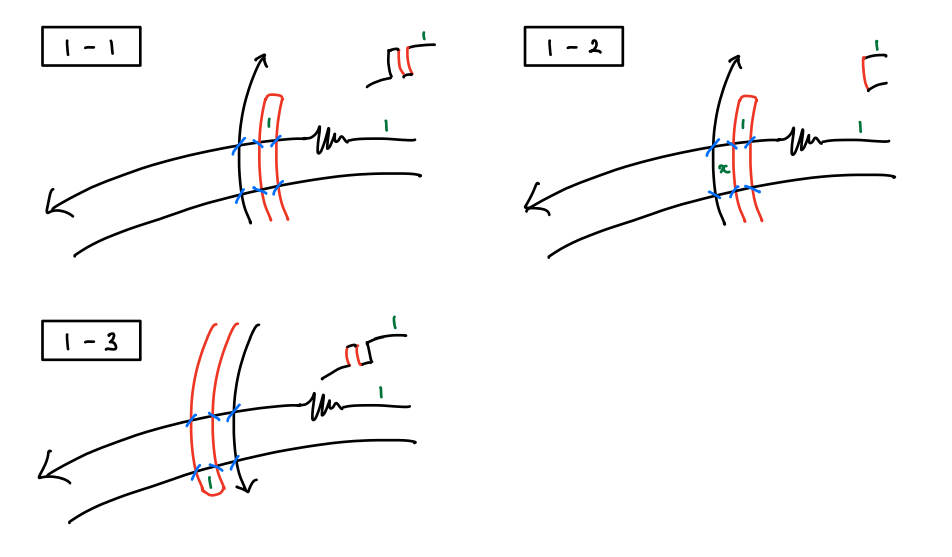}
    \caption{There are three ways a tentacle is introduced to the local picture in Case 1 of Figure \ref{farElemma0}. Note that these are the only cases we have to consider. Else, the tentacle will not interact with the relevant component in the local picture, or will be killed by $Kh(S\circ H'\circ E'')$.}
    \label{farElemma1}
    \end{figure}
    \FloatBarrier

    \begin{figure}[!htbp]
    \centering
    \includegraphics[width=0.65\textwidth]{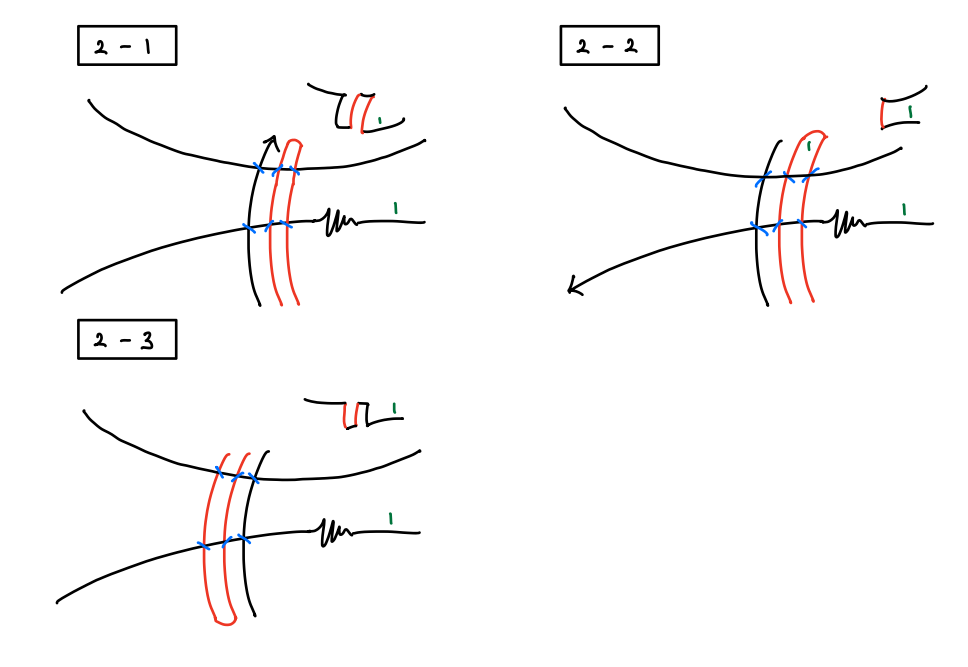}
    \caption{There are three ways a tentacle is introduced to the local picture in Case 2 of Figure \ref{farElemma0}.}
    \label{farElemma2}
    \end{figure}
    \FloatBarrier

    \begin{figure}[!htbp]
    \centering
    \includegraphics[width=0.65\textwidth]{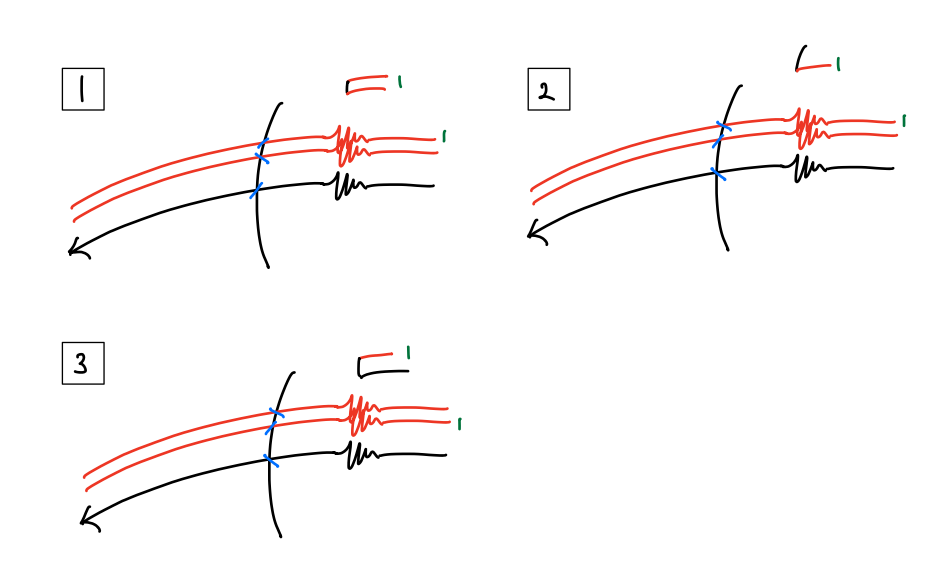}
    \caption{The strand labeled $1$ escapes the local picture of Proposition \ref{localElemma} along a strand of the tentacle (drawn red), and encounters a crossing far away from the local picture. There are three cases to consider.}
    \label{farElemmaR0}
    \end{figure}
    \FloatBarrier

    \begin{figure}[!htbp]
    \centering
    \includegraphics[width=0.65\textwidth]{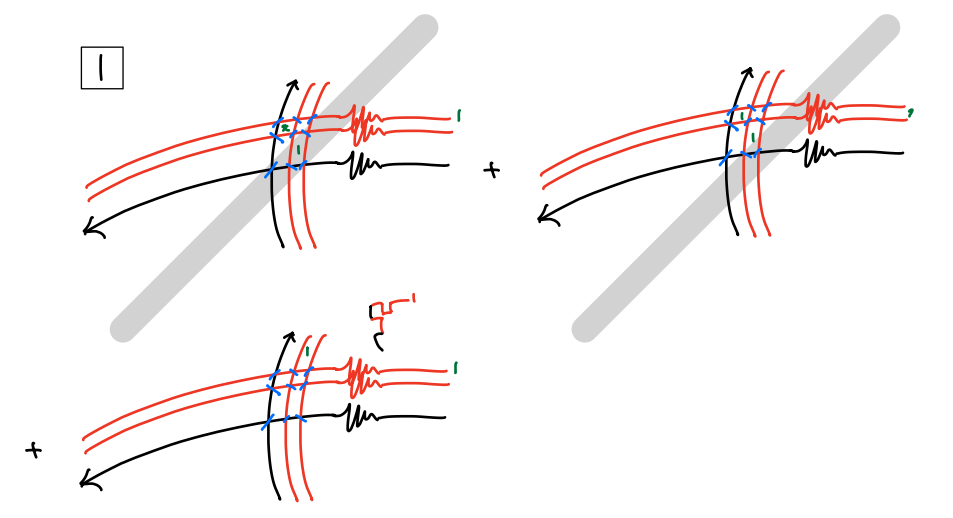}
    \caption{Case 1 of Figure \ref{farElemmaR0}. By Lemma \ref{kill1}, the summands that have been struck out (in gray) are killed by $Kh(S\circ H'\circ E'')$.}
    \label{farElemmaR1}
    \end{figure}
    \FloatBarrier

    \begin{figure}[!htbp]
    \centering
    \includegraphics[width=0.65\textwidth]{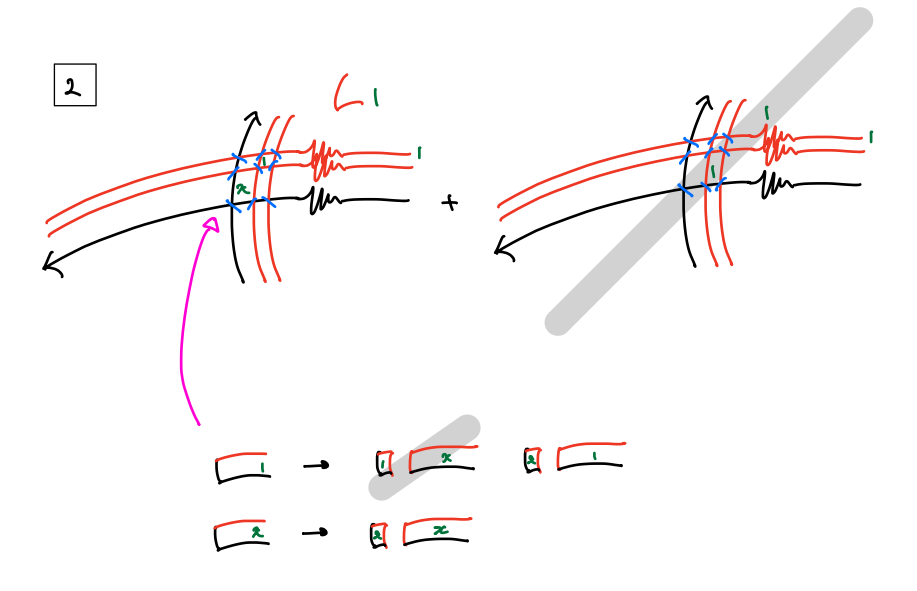}
    \caption{Case 2 of Figure \ref{farElemmaR0}. The bottom diagram justifies why we need only consider the summand on the left.}
    \label{farElemmaR2}
    \end{figure}
    \FloatBarrier

    \begin{figure}[!htbp]
    \centering
    \includegraphics[width=0.65\textwidth]{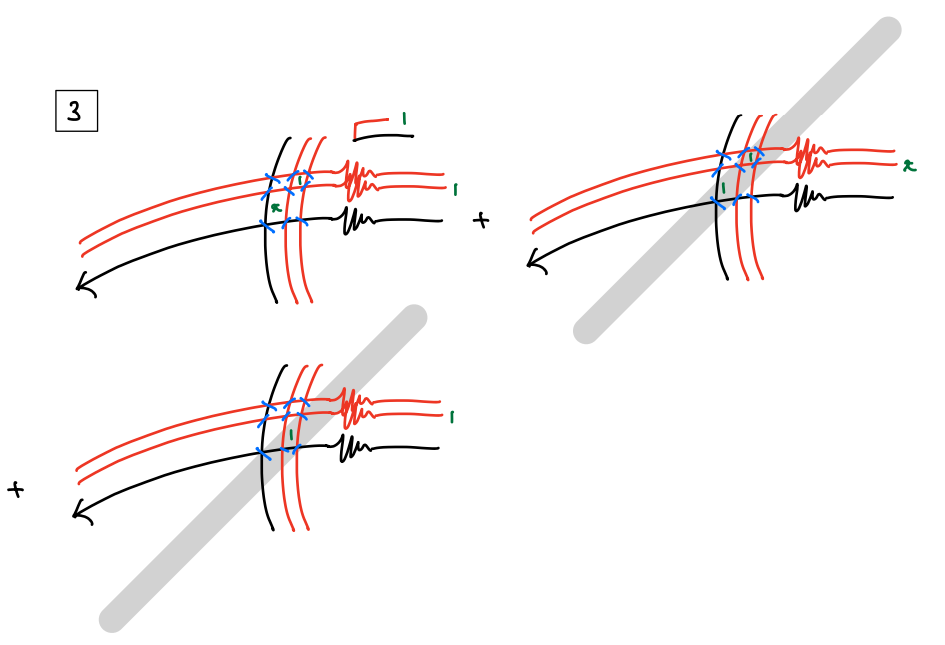}
    \caption{Case 3 of Figure \ref{farElemmaR0}.}
    \label{farElemmaR3}
    \end{figure}
    \FloatBarrier

\subsection{Proof of Lemma \ref{farEmovelemmax}}
    The proof is completely analogous to that of Lemma \ref{farEmovelemma1}. The relevant computations are detailed in Figures \ref{farElemmax0}-\ref{farElemmaxR3}.

    \begin{figure}[!htbp]
    \centering
    \includegraphics[width=0.65\textwidth]{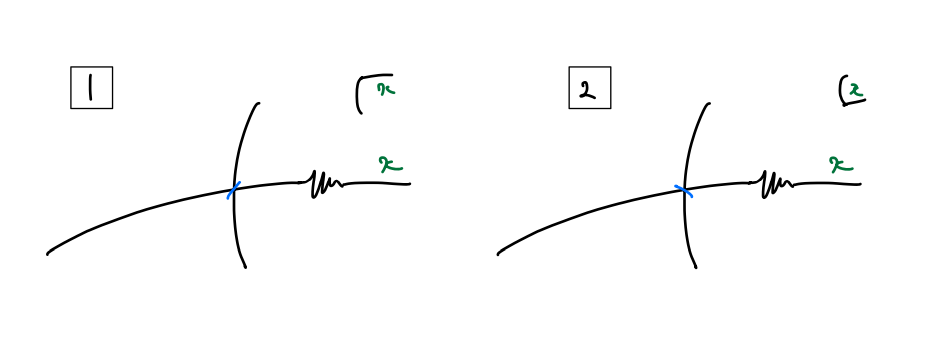}
    \caption{The strand labeled $x$ escapes the local picture of Proposition \ref{localElemma} along a strand of $K$ (drawn black), and encounters a crossing far away from the local picture. There are two ways this crossing can be resolved.}
    \label{farElemmax0}
    \end{figure}
    \FloatBarrier

    \begin{figure}[!htbp]
    \centering
    \includegraphics[width=0.65\textwidth]{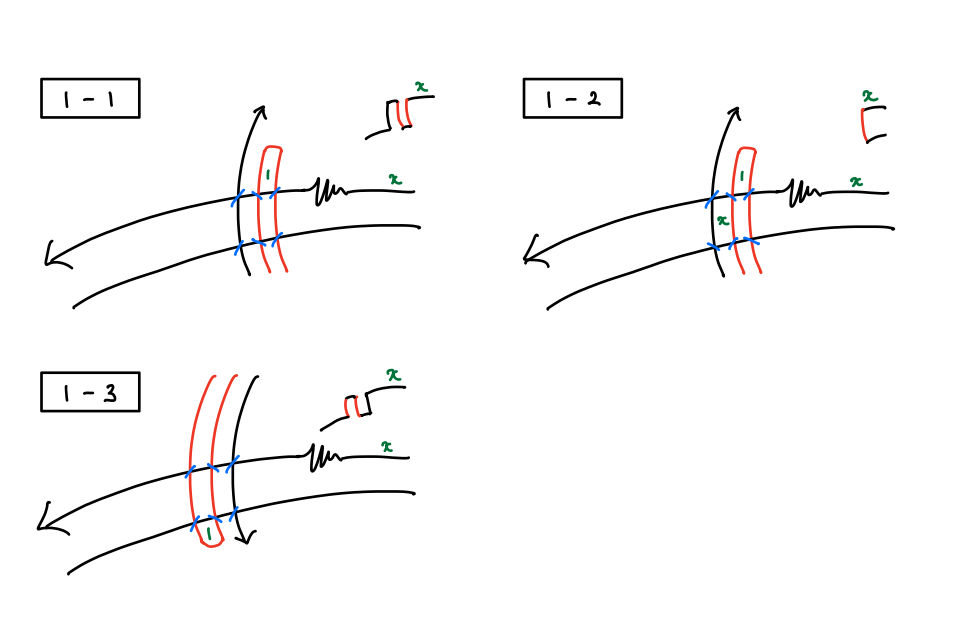}
    \caption{There are three ways a tentacle is introduced to the local picture in Case 1 of Figure \ref{farElemmax0}. Note that these are the only cases we have to consider. Else, the tentacle will not interact with the relevant component in the local picture.}
    \label{farElemmax1}
    \end{figure}
    \FloatBarrier

    \begin{figure}[!htbp]
    \centering
    \includegraphics[width=0.65\textwidth]{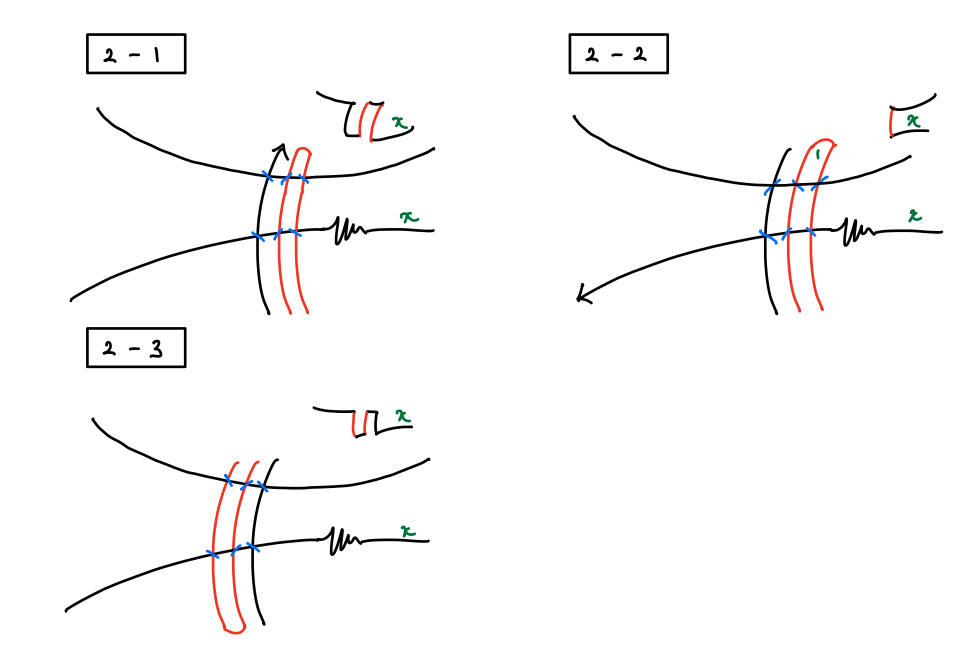}
    \caption{There are three ways a tentacle is introduced to the local picture in Case 2 of Figure \ref{farElemmax0}.}
    \label{farElemmax2}
    \end{figure}
    \FloatBarrier

    \begin{figure}[!htbp]
    \centering
    \includegraphics[width=0.65\textwidth]{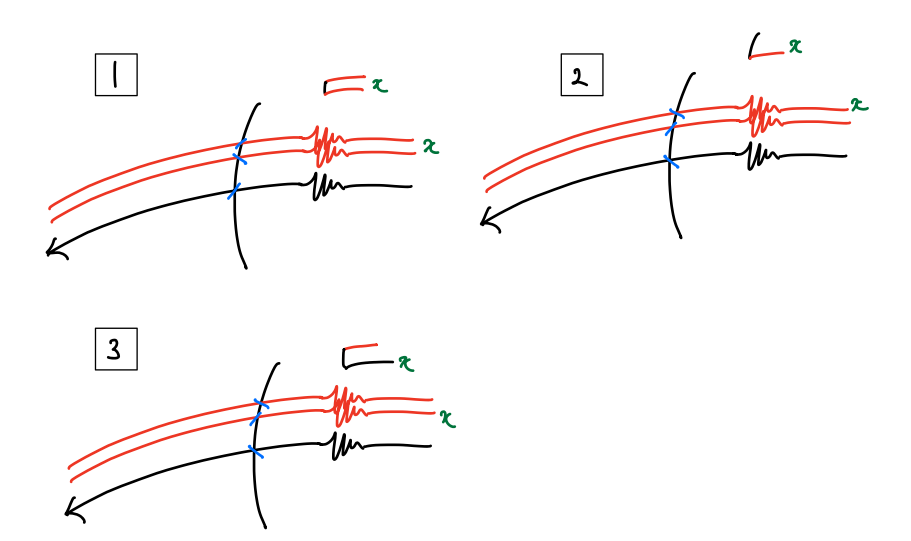}
    \caption{The strand labeled $x$ escapes the local picture of Proposition \ref{localElemmax} along a strand of the tentacle (drawn red), and encounters a crossing far away from the local picture. There are three cases to consider.}
    \label{farElemmaxR0}
    \end{figure}
    \FloatBarrier

    \begin{figure}[!htbp]
    \centering
    \includegraphics[width=0.65\textwidth]{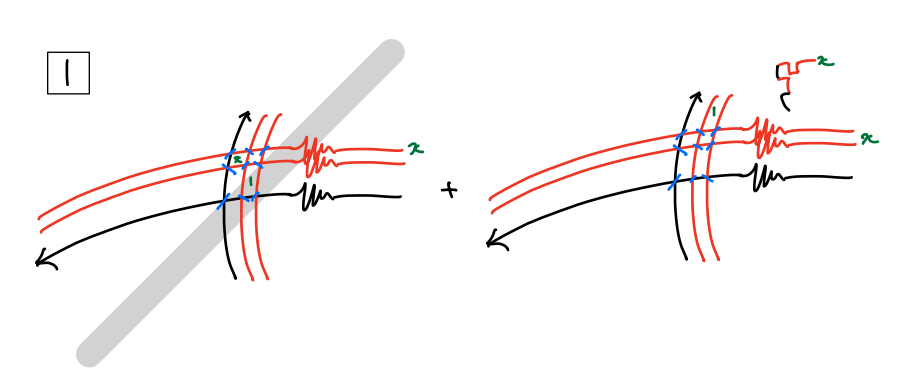}
    \caption{Case 1 of Figure \ref{farElemmaxR0}. By Lemma \ref{kill1}, the summands that have been struck out (in gray) are killed by $Kh(S\circ H'\circ E'')$.}
    \label{farElemmaxR1}
    \end{figure}
    \FloatBarrier

    \begin{figure}[!htbp]
    \centering
    \includegraphics[width=0.65\textwidth]{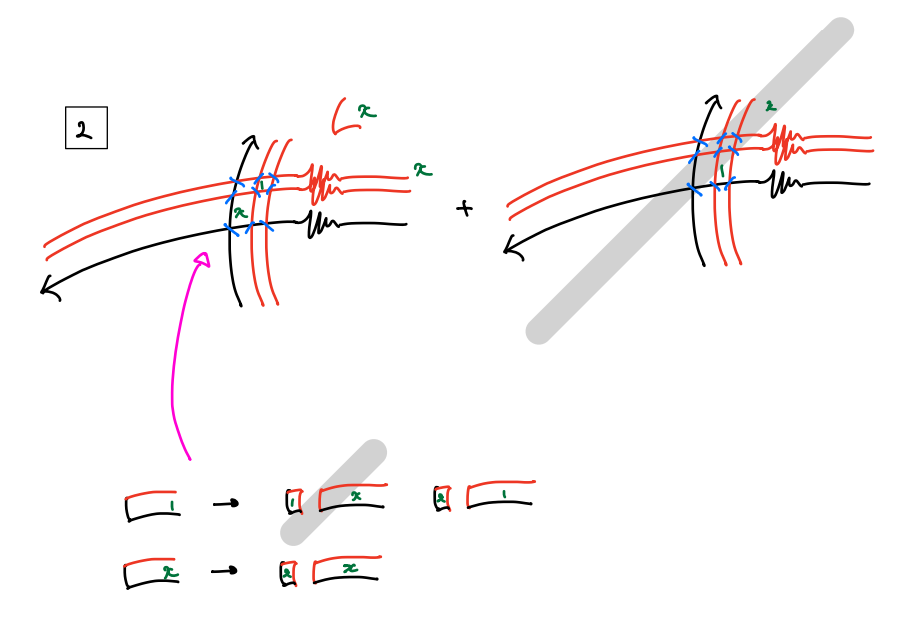}
    \caption{Case 2 of Figure \ref{farElemmaxR0}. The bottom diagram justifies why we need only consider the summand on the left.}
    \label{farElemmaxR2}
    \end{figure}
    \FloatBarrier

    \begin{figure}[!htbp]
    \centering
    \includegraphics[width=0.65\textwidth]{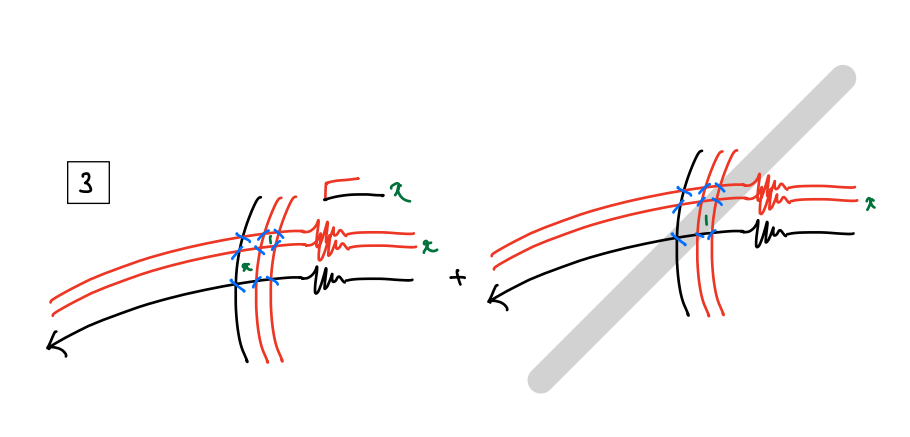}
    \caption{Case 3 of Figure \ref{farElemmaxR0}.}
    \label{farElemmaxR3}
    \end{figure}
    \FloatBarrier

\subsection{Proof of Lemma \ref{beginlemma}}

    As described in Figure \ref{beginlemma0}, there are four possible cases, depending on how the first crossing of $K$ is resolved in $a$, and how $Q$ is labeled. The computations detailed in Figures \ref{beginlemma1}-\ref{beginlemma4'} verify the claim.\qed

    \begin{figure}[!htbp]
    \centering
    \includegraphics[width=0.65\textwidth]{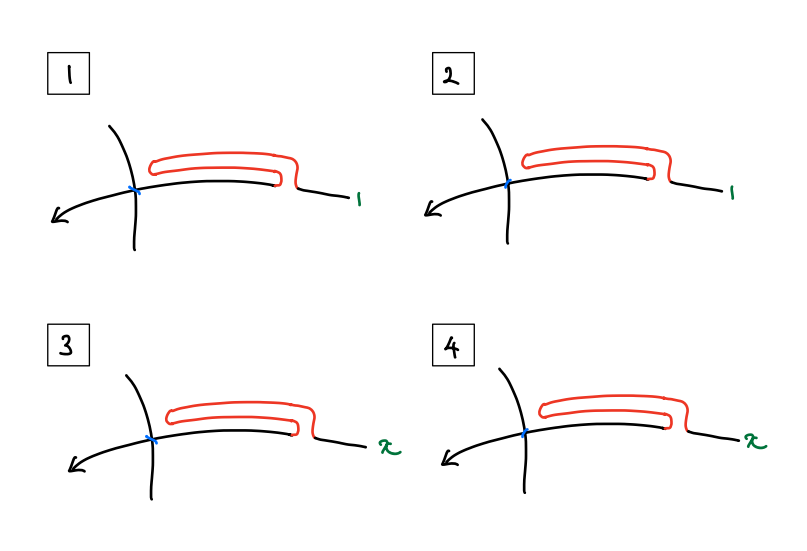}
    \caption{local picture of around the first crossing of $K$. There are four cases depending on how the crossing is resolved, and how the component containing the attaching region in is labeled in $a$.}
    \label{beginlemma0}
    \end{figure}
    \FloatBarrier

    \begin{figure}[!htbp]
    \centering
    \includegraphics[width=0.65\textwidth]{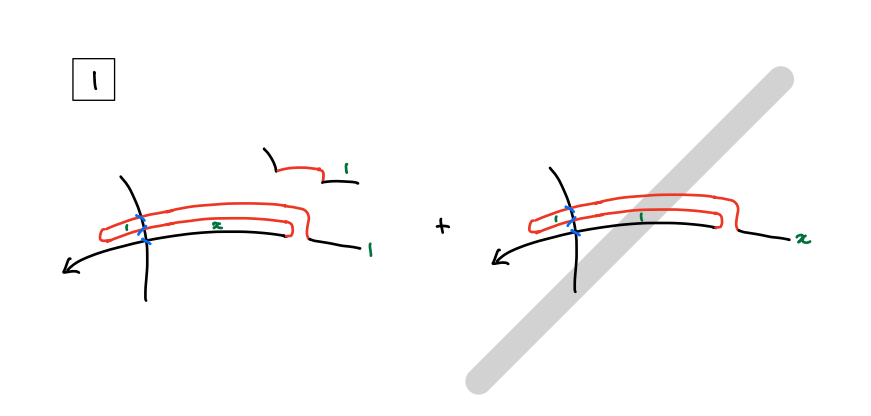}
    \caption{Case $1$ of Figure \ref{beginlemma0} after the tentacle traverses along the crossing.}
    \label{beginlemma1}
    \end{figure}
    \FloatBarrier

    \begin{figure}[!htbp]
    \centering
    \includegraphics[width=0.65\textwidth]{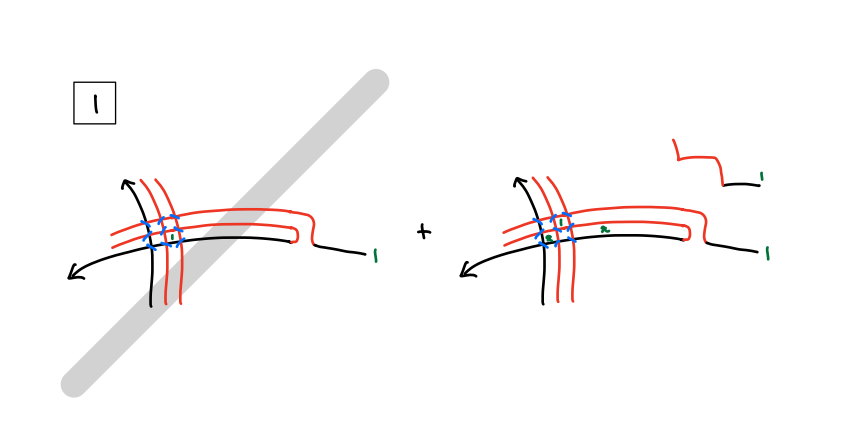}
    \caption{Case $1$ of Figure \ref{beginlemma0} after the tentacle traverses over the crossing.}
    \label{beginlemma1_}
    \end{figure}
    \FloatBarrier

    \begin{figure}[!htbp]
    \centering
    \includegraphics[width=0.65\textwidth]{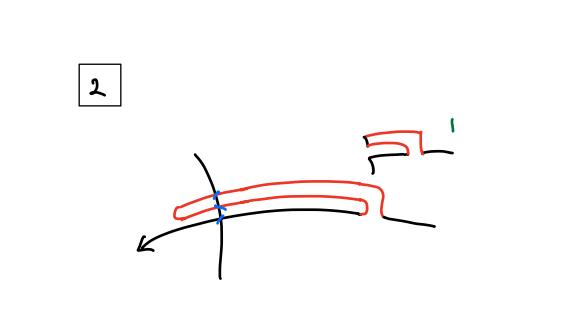}
    \caption{Case $2$ of Figure \ref{beginlemma0} after the tentacle traverses along the crossing.}
    \label{beginlemma2}
    \end{figure}
    \FloatBarrier

    \begin{figure}[!htbp]
    \centering
    \includegraphics[width=0.65\textwidth]{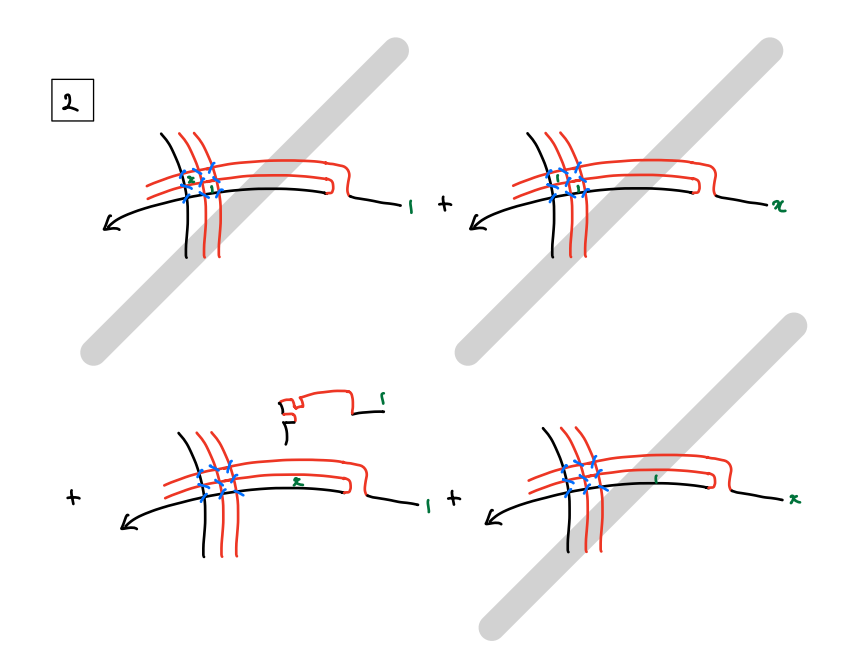}
    \caption{Case $2$ of Figure \ref{beginlemma0} after the tentacle traverses over the crossing.}
    \label{beginlemma2_}
    \end{figure}
    \FloatBarrier

    \begin{figure}[!htbp]
    \centering
    \includegraphics[width=0.4\textwidth]{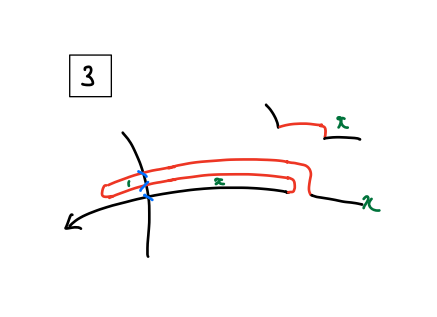}
    \caption{Case $3$ of Figure \ref{beginlemma0} after the tentacle traverses along the crossing.}
    \label{beginlemma3}
    \end{figure}
    \FloatBarrier

    \begin{figure}[!htbp]
    \centering
    \includegraphics[width=0.65\textwidth]{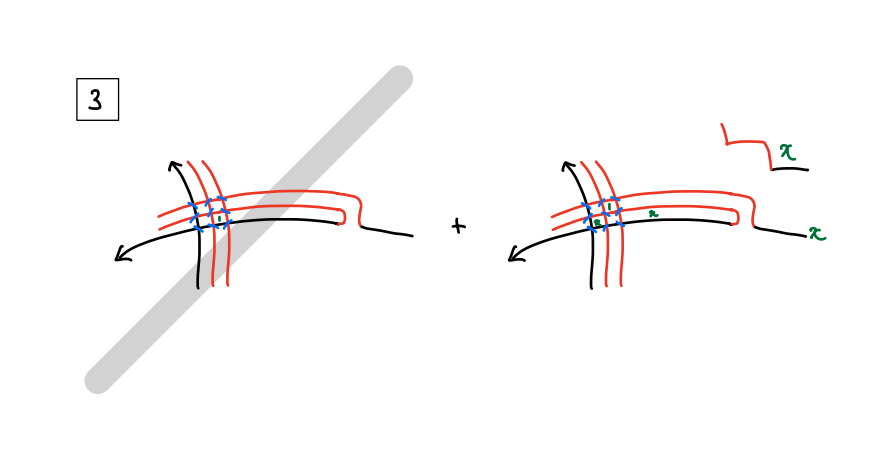}
    \caption{Case $3$ of Figure \ref{beginlemma0} after the tentacle traverses over the crossing.}
    \label{beginlemma3_}
    \end{figure}
    \FloatBarrier

    \begin{figure}[!htbp]
    \centering
    \includegraphics[width=0.4\textwidth]{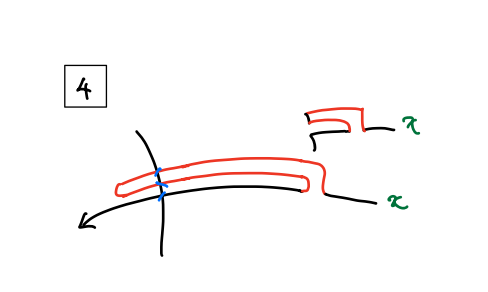}
    \caption{Case $4$ of Figure \ref{beginlemma0} after the tentacle traverses along the crossing.}
    \label{beginlemma4}
    \end{figure}
    \FloatBarrier

    \begin{figure}[!htbp]
    \centering
    \includegraphics[width=0.65\textwidth]{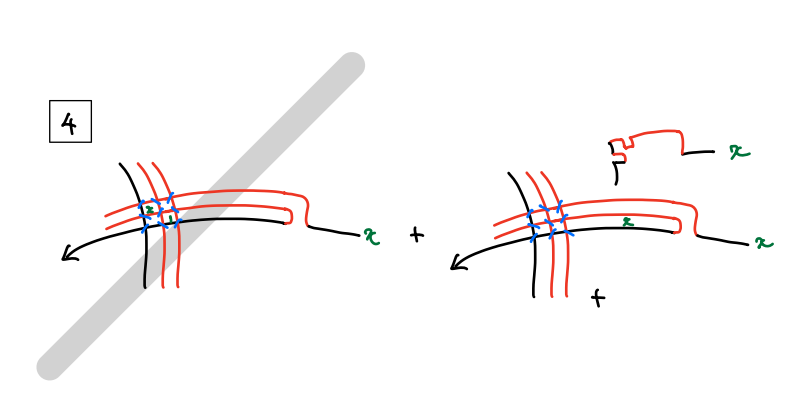}
    \caption{Case $4$ of Figure \ref{beginlemma0} after the tentacle traverses over the crossing.}
    \label{beginlemma4'}
    \end{figure}
    \FloatBarrier

\subsection{Proof of Lemma \ref{endlemma}}

    Figure \ref{endlemma0} describes the relevant local picture. There are two cases to consider, depending on how the last crossing of $K$ is resolved. Figure \ref{endlemma1} verify the claim in each of these two cases. \qed

    \begin{figure}[!htbp]
    \centering
    \includegraphics[width=0.5\textwidth]{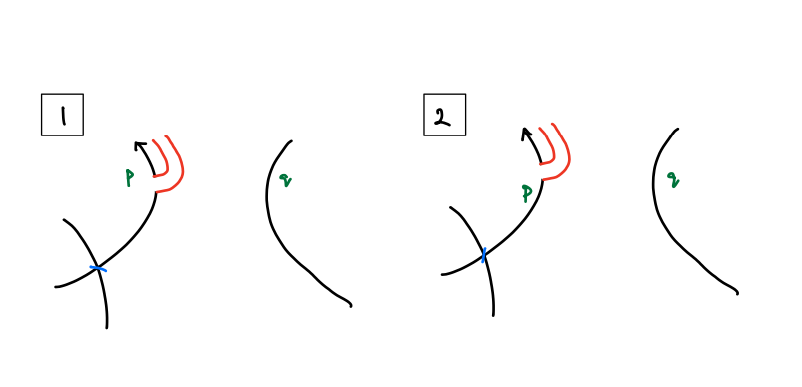}
    \caption{local picture of $K$ and $\bar K$ around the attaching region for $H$. The components of $K$, $\bar K$ containing the connected components of the attaching region are labeled $p$, $q$ respectively. There are two ways the final crossing is resolved in $a$, contributing to two cases.}
    \label{endlemma0}
    \end{figure}
    \FloatBarrier

    \begin{figure}[!htbp]
    \centering
    \includegraphics[width=0.5\textwidth]{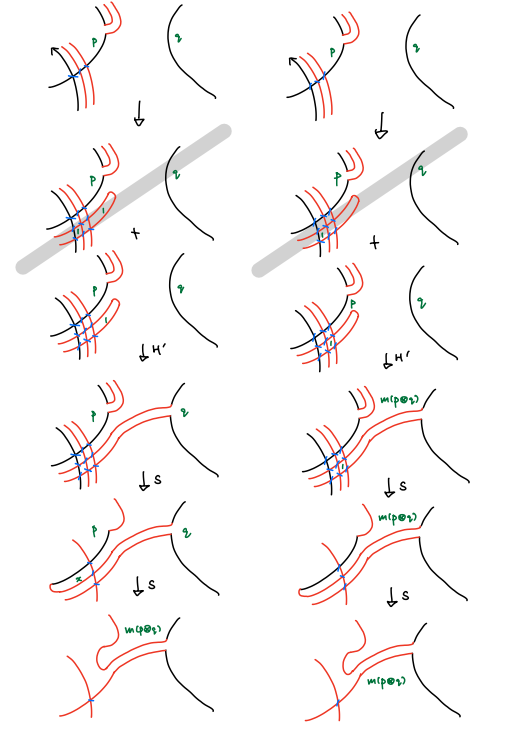}
    \caption{The two cases of Figure \ref{endlemma0}.}
    \label{endlemma1}
    \end{figure}
    \FloatBarrier

\subsection{Figures for proof of the Theorem \ref{mainresult}}

    \begin{figure}[!htbp]
    \centering
    \includegraphics[width=0.65\textwidth]{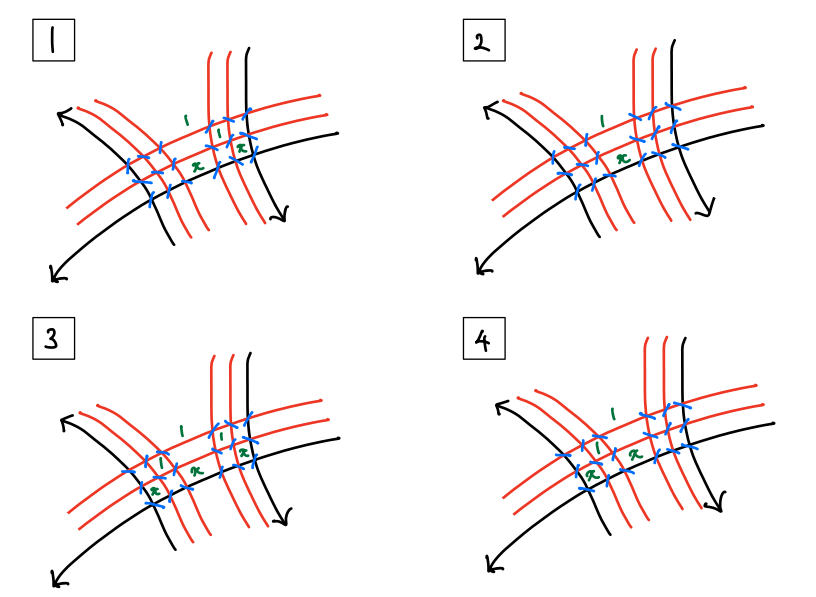}
    \caption{The four possible cases for local pictures of $b'$, when $Q$ is labeled $1$. Omitted are twelve other cases, corresponding to changes in the orientations of the vertical strands of $K$.}
    \label{localSlemma0-1}
    \end{figure}
    \FloatBarrier

    \begin{figure}[!htbp]
    \centering
    \includegraphics[width=0.65\textwidth]{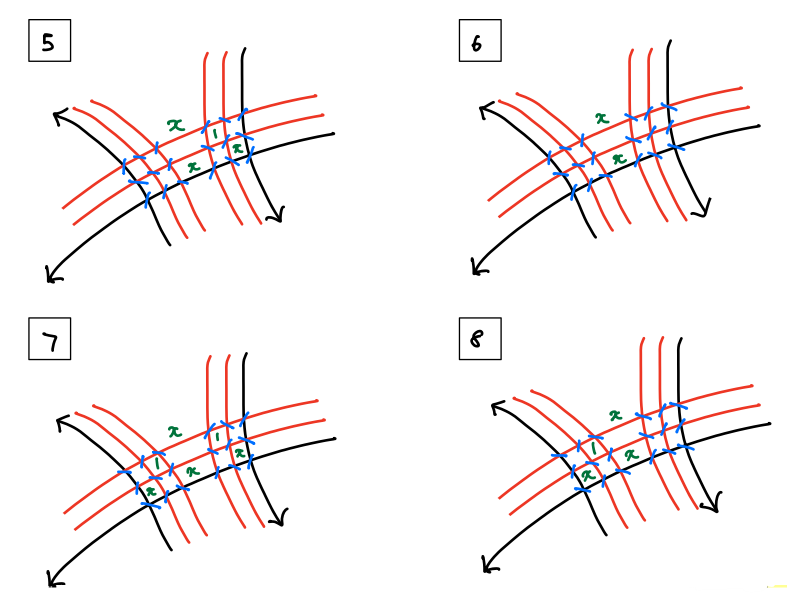}
    \caption{The four possible cases for local pictures of $b'$, when $Q$ is labeled $x$. Omitted are twelve other cases, corresponding to changes in the orientations of the vertical strands of $K$.}
    \label{localSlemma0-2}
    \end{figure}
    \FloatBarrier

    \begin{figure}[!htbp]
    \centering
    \includegraphics[width=0.65\textwidth]{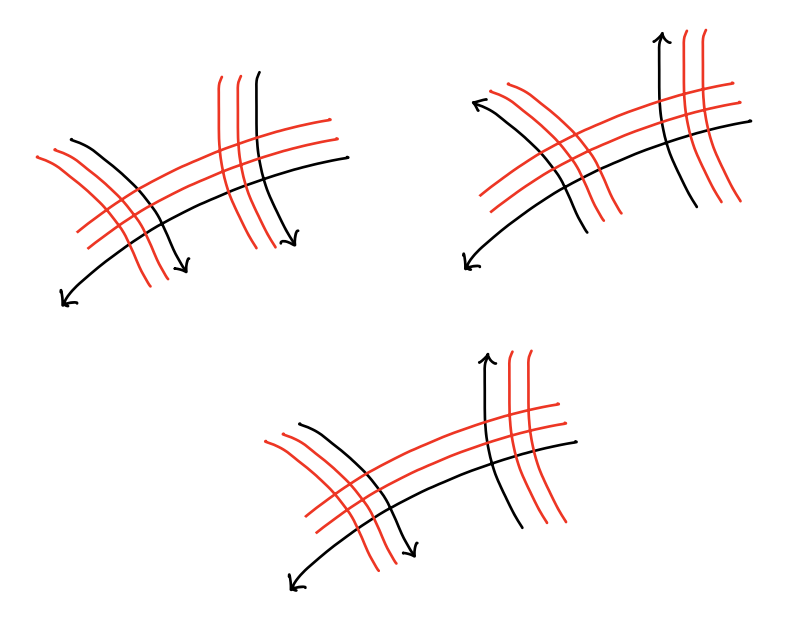}
    \caption{There are three other choices of orientations for the vertical strands of $K$ in the local picture. Each choice contributes eight cases, depending on how the crossings $c_1$ and $c_2$ are resolved.}
    \label{localSlemma0'}
    \end{figure}
    \FloatBarrier

    \begin{figure}[!htbp]
    \centering
    \includegraphics[width=0.65\textwidth]{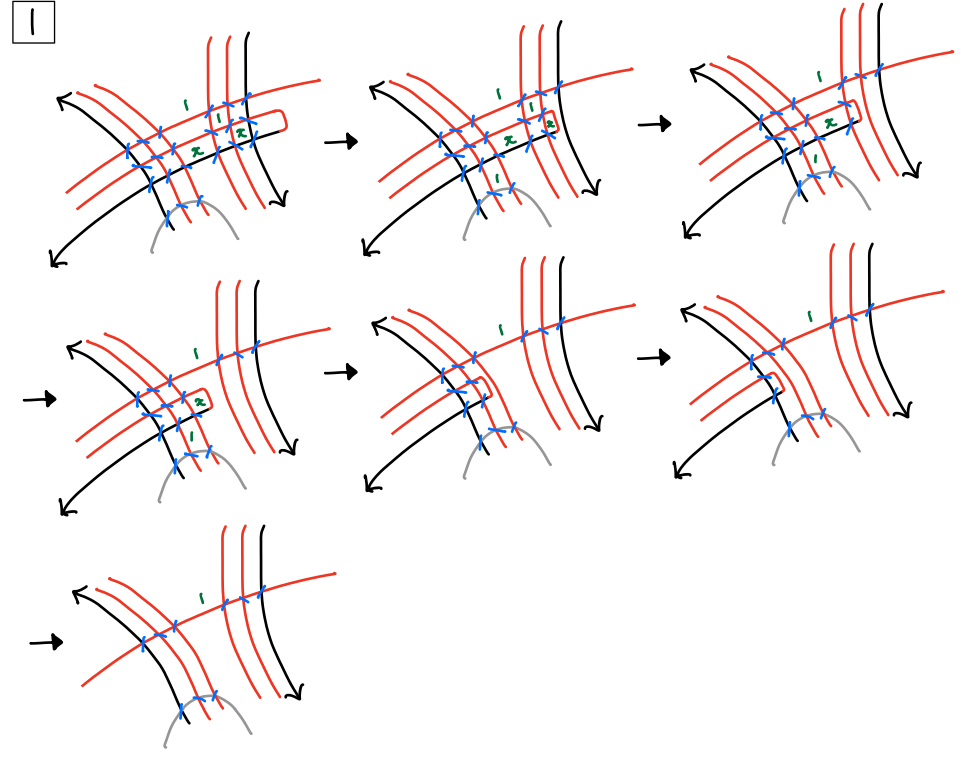}
    
    \label{localSlemma1}
    \end{figure}
    \FloatBarrier

    \begin{figure}[!htbp]
    \centering
    \includegraphics[width=0.65\textwidth]{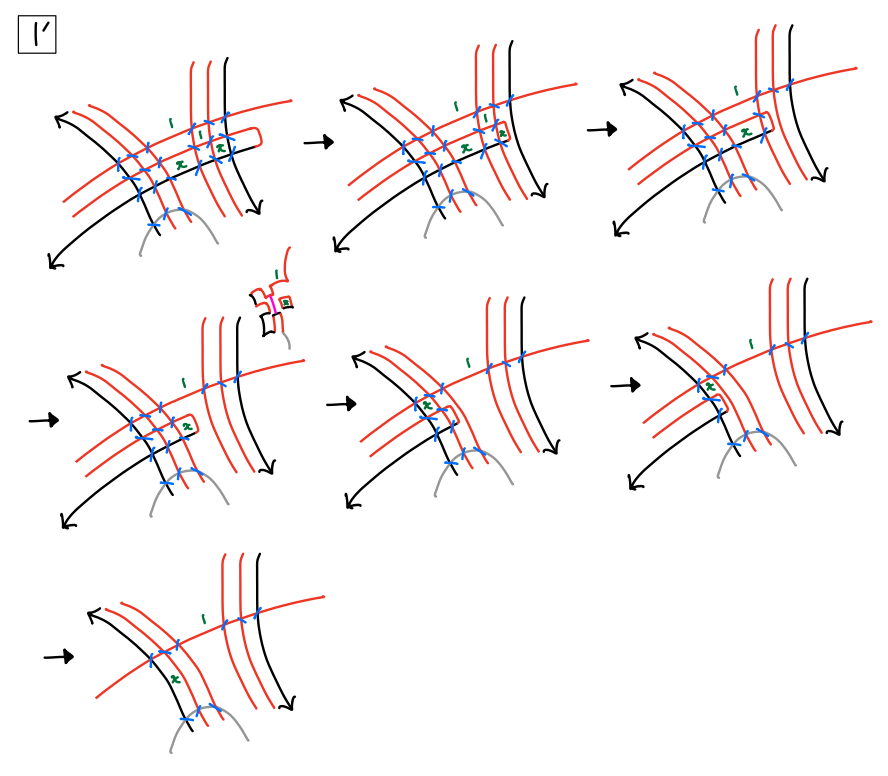}
    \label{localSlemma1'}
    \end{figure}
    \FloatBarrier

    \begin{figure}[!htbp]
    \centering
    \includegraphics[width=0.65\textwidth]{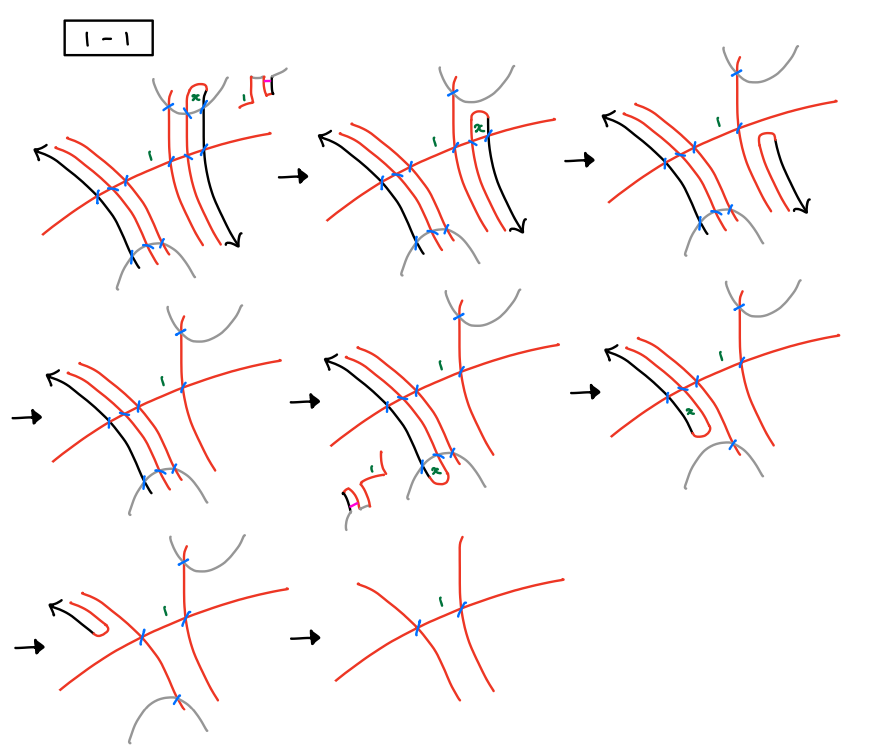}
    \label{localSlemma1-1}
    \end{figure}
    \FloatBarrier

    \begin{figure}[!htbp]
    \centering
    \includegraphics[width=0.65\textwidth]{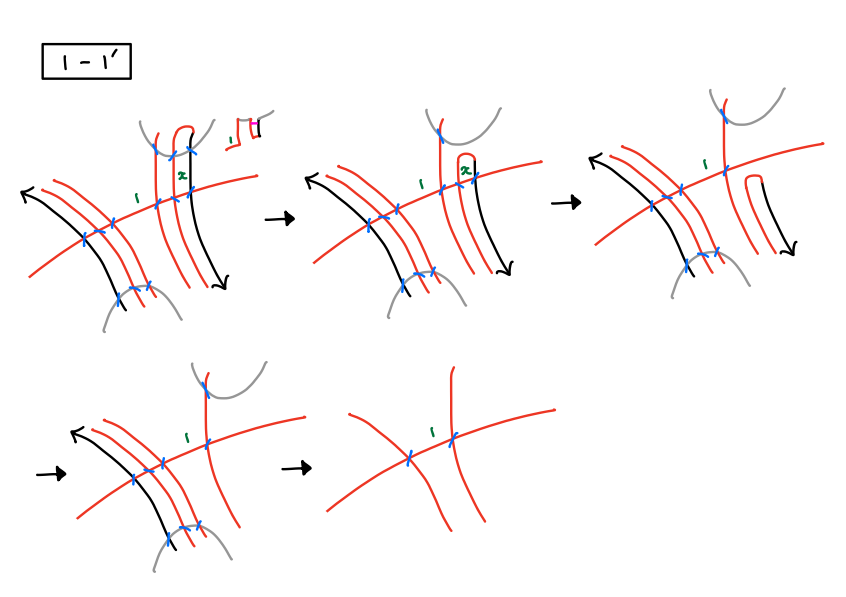}
    \label{localSlemma1-1'}
    \end{figure}
    \FloatBarrier

    \begin{figure}[!htbp]
    \centering
    \includegraphics[width=0.65\textwidth]{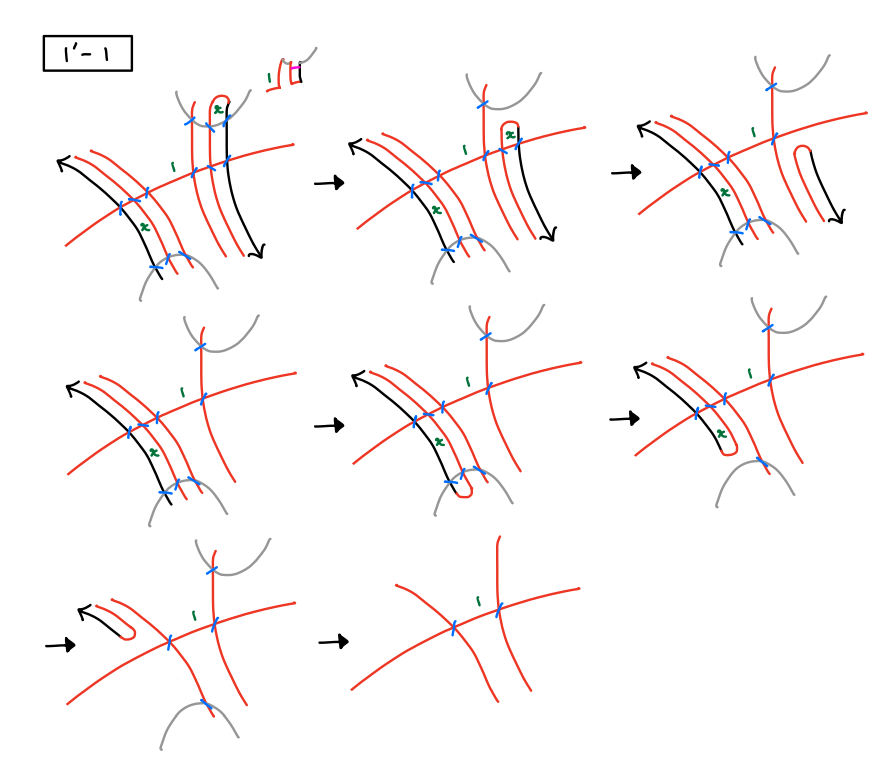}
    \label{localSlemma1'-1}
    \end{figure}
    \FloatBarrier

    \begin{figure}[!htbp]
    \centering
    \includegraphics[width=0.65\textwidth]{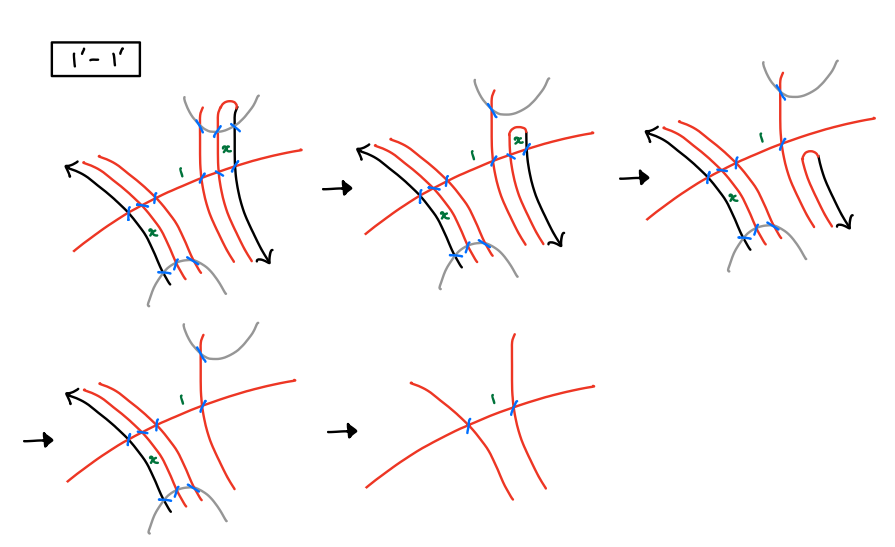}
    \caption{Verifying Case 1 of Figure \ref{localSlemma0-1}. An auxiliary arc (gray), that can be either a strand of $K$ or a strand of the tentacle, is introduced in order to compute the effect of $CKh(S)$ on the local picture. There are two ways the crossings on the auxiliary arc are resolved, leading to intermediate cases 1 and 1'.}
    \label{localSlemma1'-1'}
    \end{figure}
    \FloatBarrier

    \begin{figure}[!htbp]
    \centering
    \includegraphics[width=0.65\textwidth]{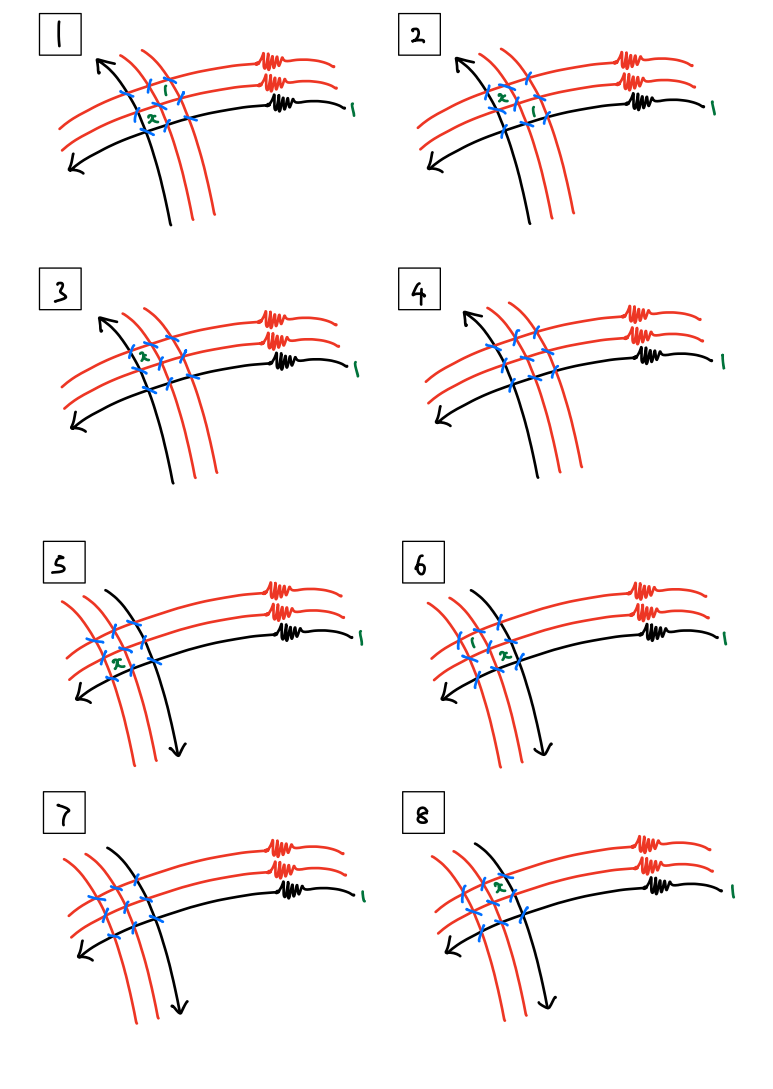}
    \caption{There are eight possible local pictures of a strand of $K$ labeled $1$ encountering a crossing, depending on the orientation of the vertical strand of $K$, and the resolution of the middle crossing.}
    \label{farSlemma0-1}
    \end{figure}
    \FloatBarrier

    \begin{figure}[!htbp]
    \centering
    \includegraphics[width=0.65\textwidth]{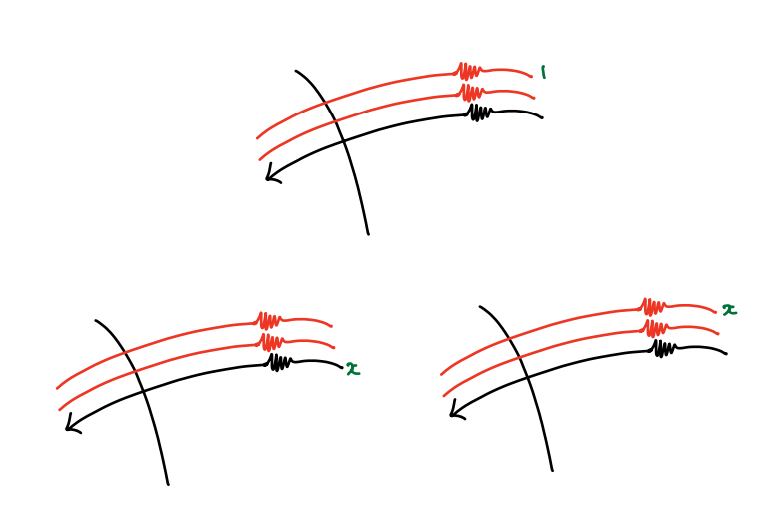}
    \caption{There are three other cases to be considered, each contributing to eight possible local pictures.}
    \label{farSlemma0'}
    \end{figure}
    \FloatBarrier
    
    \begin{figure}[!htbp]
    \centering
    \includegraphics[width=0.65\textwidth]{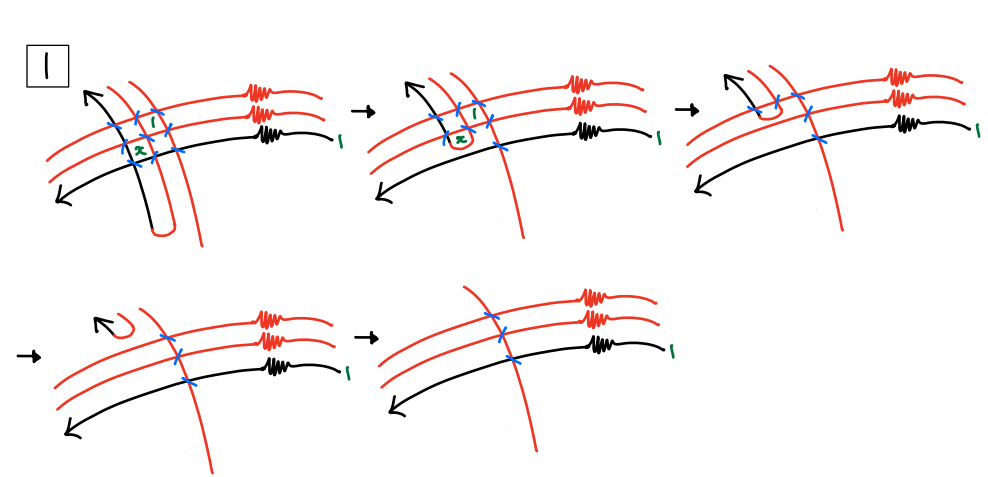}
    \caption{Verifying Case 1 of Figure \ref{farSlemma0-1}. The remaining cases are verified analogously, and are omitted.}
    \label{farSlemma1}
    \end{figure}
    \FloatBarrier

\end{document}